\documentclass[a4paper,11pt] {amsart}
\usepackage{amssymb,latexsym,amsmath}
\usepackage{amsfonts,amsbsy,bm}
\def\a {\alpha }

\renewcommand {\phi} {{\varphi}}

\newcommand {\al} {{\alpha}}

\newcommand {\da} {{\delta}}
\newcommand {\Da} {{\Delta}}

\newcommand {\la} {{\lambda}}

\newcommand {\cD} {{\mathcal D}}

\newcommand {\N} {{\mathbb N}}

\newcommand {\Scirc} {\raise.2ex\hbox{$\scriptstyle\circ$}}

\newcommand {\im} {{\Im\!\mbox{\small\it m}\,}}

\newcommand {\mand} {{\quad\mbox{and}\quad}}
\renewcommand {\mid} {{\,\,\,\colon\,\,\,}}

\newcommand {\Proof} {\noindent{\bf P{\footnotesize\bf ROOF}: } \ }
\newcommand {\Proofof}[1] {\noindent{\bf P{\footnotesize\bf ROOF} of {#1}: } \ }
\newcommand {\ProofEnd} {
             \begin{flushright} \vskip -0.2in $\Box$ \end{flushright}}

\newcommand{\Ba}[1]{\begin{array}{#1}}
\newcommand{\Ea}{\end{array}}
\newcommand{\Be}{\begin{equation}}
\newcommand{\Ee}{\end{equation}}
\newcommand{\Bea}{\begin{eqnarray}}
\newcommand{\Eea}{\end{eqnarray}}
\newcommand{\Beas}{\begin{eqnarray*}}
\newcommand{\Eeas}{\end{eqnarray*}}
\newcommand{\Benu}{\begin{enumerate}}
\newcommand{\Eenu}{\end{enumerate}}
\newcommand{\Bi}{\begin{itemize}}
\newcommand{\Ei}{\end{itemize}}

\newcommand{\BR}{\begin{Remark} \em}
\newcommand{\ER}{\end{Remark}}
\newcommand{\BE}{\begin{example} \em}
\newcommand{\EE}{\end{example}}
\newcounter{remark}

\newtheorem{theorem}[equation]{T{\hskip 0pt\footnotesize\bf HEOREM}}
\newtheorem{proposition}[equation]{P{\hskip 0pt\footnotesize\bf ROPOSITION}}
\newtheorem{corollary}[equation]{C{\hskip 0pt\footnotesize\bf OROLLARY}}
\newtheorem{lemma}[equation]{L{\hskip 0pt\footnotesize\bf EMMA}}
\newtheorem{Remark}[equation]{R{\hskip 0pt\footnotesize\bf EMARK}}

\newtheorem{example}[equation]{E{\hskip 0pt\footnotesize\bf XAMPLE}}

\newcommand{\bprop} {\begin{proposition}}
\newcommand{\eprop} {\end{proposition}}
\newcommand{\btheo} {\begin{theorem}}
\newcommand{\etheo} {\end{theorem}}
\newcommand{\blem} {\begin{lemma}}
\newcommand{\elem} {\end{lemma}}
\newcommand{\bcor} {\begin{corollary}}
\newcommand{\ecor} {\end{corollary}}

scaled \magstep1

\newcommand{\be}{{\bf e}}

\newcommand{{\tBox}}{{\widetilde{\raisebox{-0.2ex}[1.25ex][0ex]{$\Box$}}}}

\begin{document}
\title[Toeplitz and Hankel operators from Bergman to analytic Besov spaces]{Toeplitz and Hankel operators from Bergman to analytic Besov spaces of tube
  domains over symmetric cones.\footnotetext{\emph{2000 Math Subject Classification:} 42B35, 32M15.}
}
\author{Cyrille Nana}
\address{Cyrille Nana, Department of Mathematics, Faculty of Science, University of Buea, P. O. Box 63, Buea Cameroon}
\email{nana.cyrille@ubuea.cm}
\author{Beno\^it F. Sehba}
\address{Beno\^it Florent Sehba, Department of Mathematics, University of Ghana, P. O. Box LG 62, Legon, Accra, Ghana}
\email{bfsehba@ug.edu.gh}
 \maketitle
\begin{abstract}
We characterize bounded Toeplitz and Hankel operators from weighted Bergman spaces to weighted Besov spaces in tube domains over symmetric cones. We deduce weak factorization results for some Bergman spaces of this setting.
\end{abstract}

\section{Introduction}
\setcounter{equation}{0} \setcounter{footnote}{0}
\setcounter{figure}{0} Let $n\geq 3$ and $\mathcal D=\mathbb R^n+i\Omega$
be the tube domain over an irreducible symmetric  cone $\Omega$
in $\mathbb R^n$. Following the notations of \cite{FK} we
denote the rank of the cone $\Omega$ by $r$ and by $\Delta$
the determinant function of $\mathbb R^n$. As example of
symmetric cone on $\mathbb R^n$ we have the Lorentz cone
$\Lambda_n$ which is a rank 2 cone defined for $n\ge 3$ by
$$\Lambda_n=\{(y_1,\cdots,y_n)\in \mathbb R^n: y_1^2-\cdots-y_n^2>0,\,\,\,y_1>0\},$$
the determinant function in this case is given by the
Lorentz form
$$\Delta(y)=y_1^2-\cdots-y_n^2.$$  We shall denote by $\mathcal{H}(\mathcal D)$ the space
of holomorphic functions on $\mathcal D$.

\vskip .2cm
 For $1\le p<\infty$ and $\nu \in \mathbb R$, let
$L^p_\nu(\mathcal D)=L^p(\mathcal
D,\Delta^{\nu-\frac{n}{r}}(y)dx\,dy)$
denotes the space of functions $f$ 
satisfying the condition
$$\|f\|_{p,\nu}=||f||_{L^p_\nu(\mathcal D)}:=\left(\int_{\mathcal D}|f(x+iy)|^p\Delta^{\nu-\frac{n}{r}}(y)dxdy\right)^{1/p}<\infty.$$
Its closed subspace consisting of holomorphic functions in $\mathcal
D$ is the weighted Bergman space $A^p_\nu(\mathcal D)$. This space
is not trivial i.e. $A^p_\nu(\mathcal D)\neq \{0\}$ only for
$\nu>\frac{n}{r}-1$ (see \cite{DD}). 
\vskip .2cm
The weighted Bergman projection $P_\nu$ is the orthogonal projection of the Hilbert space $L^2_\nu(\mathcal D)$ onto its closed subspace $A^2_\nu(\mathcal D).$ It is well known that $P_\nu$ is an integral operator given by $$P_\nu f(z)=\int_{\mathcal D}K_\nu(z, w) f(w)
dV_\nu(w),
$$
where
 $K_\nu(z, w)=
 c_\nu\,\Delta^{-(\nu+\frac{n}{r})}((z-\overline {w})/i)$
is the weighted Bergman kernel, i.e. the reproducing kernel of  $A^2_\nu(\mathcal D)$ (see
\cite{FK}). Here, we use the notation
$dV_\nu(w):=\Delta^{\nu-\frac{n}{r}}(v) du\,dv$, where
$w=u+iv$ is an element of $\mathcal D$. The unweighted case corresponds to $\nu=\frac{n}{r}.$

\vskip .2cm
It is known that the weighted Bergman projection $P_\nu$ cannot be bounded on $L^p_\nu(\mathcal D)$ for $p$ large
(see \cite{BBGR}, \cite{sehba}). Consequently, the natural mapping
from $A^{p'}_\nu(\mathcal D)$ ($\frac{1}{p}+\frac{1}{p'}=1$) into the dual space $(A^p_\nu(\mathcal
D))^*$ of $A^p_\nu(\mathcal D)$ is not an isomorphism for such
values of the exponent $p$ (see \cite{BBGRS}). In \cite{BBGRS}, it is
also shown that the boundedness of $P_\nu$ on $L^p_\nu(\mathcal D)$ for $p\ge 2$
is equivalent to the validity of the following Hardy-type inequality
for $f\in A^p_\nu(\mathcal D)$:
\begin{equation}\label{hardy}
\int\!\!\int_{\mathcal D} |f(x+iy)|^p \,\Delta^{\nu -
\frac{n}{r}} (y)\,dx\,dy \, \leq\,
C\,\int\!\!\int_{\mathcal D} \,\bigl|\Delta (y) \Box
 f(x+iy)\bigr|^p\,\Delta^{\nu - \frac{n}{r}} (y)\,dx\,dy,
\end{equation} where $\Box=\Delta(\frac{1}{i}
\frac{\partial}{\partial x})$ denotes the differential operator of
degree $r$ in $\mathbb R^n$ defined by the equality: \Be
\Box\,[e^{i(x|\xi)}]=\Delta(\xi)e^{i(x|\xi)}, \quad
x,\xi\in \mathbb R^n. \label{bbox} \Ee When (\ref{hardy})
holds for some $p$ and $\nu$, we speak of Hardy inequality
for $(p,\nu)$. This leads to the following definition of
analytic Besov space $\mathbb B_\nu^p(\mathcal D)$, $1\le
p<\infty$, $\nu\in \mathbb R$.

Let $m\in \mathbb N$, we denote $$\mathcal {N}_m := \{F \in
\mathcal H (\mathcal D) : \Box^m F = 0 \}$$ and set
$$\mathcal{H}_m(\mathcal D)=\mathcal{H}(\mathcal D)/ \mathcal N_m.$$ For
simplicity, we use the following notation for the
normalized operator Box operator: we write
\begin{equation}\label{normalized}
\Delta^m\Box^m F(z):=\Delta^m(\im z)\Box^mF(z).
\end{equation}
We will use the same notations for holomorphic functions
and for equivalence classes in $\mathcal H_m$.

We observe that,
for $F\in\mathcal H_m$, we can speak of the function
$\Box^m F$. Given $\nu\in\mathbb R$, $1\leq p<\infty$ the
Besov space $\mathbb B_\nu^p(\mathcal D)$ is  defined as
follows
\[ \mathbb B^p_\nu:=\bigl\{F\in\mathcal H_{k_0}(\mathcal D)\mid
\Delta^{k_0}\Box^{k_0} F\in L^p_\nu\bigr\}
\]
endowed with the norm $\|F\|_{\mathbb
B^p_\nu}=\|\Delta^{k_0}\Box^{k_0} F\|_{L^p_\nu}$. Here
$k_0=k_0(p,\mu)$ is fixed by \Be
k_0(p,\nu):=\min\bigl\{k\geq0\mid \nu+kp>\frac
{n}{r}-1\;\mbox{ and Hardy's inequality holds for
$(p,\nu+pk)$ }\bigr\}. \label{k0}\Ee Each element of
$\mathbb B^p_\nu$ is the equivalence class of all analytic
solutions of the equation $\Box^{k_0} F =g$, for some $g\in
A_{\nu+k_0p}^p.$ Consequently, these spaces are null when
$\nu+k_0p\leq \frac{n}{r}-1$. We also observe that for
$1\le p\le 2$ and $\nu>\ \frac {n}{r}-1$,  $\mathbb
B^p_\nu(\mathcal D)=A^p_\nu(\mathcal D)$ since  the Hardy
inequality (\ref{hardy}) always holds in this range (see
\cite{BBPR}). When $p>2$, we have  the embedding of
$A^p_\nu(\mathcal D)$ into $\mathbb B^p_\nu(\mathcal D)$.

\vskip .2cm
For $p=\infty$, we define the Bloch $\mathbb B=\mathbb
B^\infty$ of $\mathcal D$ as follows:
Let $m_*$ be the smallest integer $m$ such $m>\frac{n}{r}-1$, 
\[ \mathbb B:=\bigl\{F\in\mathcal H_{m_*}(\mathcal D)\mid
\Delta^{m_*}\Box^{m_*} F\in L^\infty\bigr\}.
\]

\vskip .2cm
Our interest in this paper is the boundedness of two operators, namely the Toeplitz operator with symbol a measure $\mu$ and
the (small) Hankel operator with symbol $b$ from the Bergman space $A^p_\nu(\mathcal D)$ into the Besov space $\mathbb{B}^q_\nu(\mathcal D).$
Note that usually in bounded domains as the unit ball, for the study
of the boundedness of Toeplitz and Hankel operators between Bergman spaces, one needs among others a
good understanding of the topological dual space of the target space. The choice of
$\mathbb B^q_\nu(\mathcal D)$ as the target space is
suggested not only by the lack of continuity of the
projection $P_\nu$ on $L^q_\nu(\mathcal D)$ for large
values of $q$ but also by the fact that $\mathbb B^q_\nu(\mathcal D)$ is the dual space of
of the Bergman space $A^{q'}_\nu(\mathcal D)$ ($\frac{1}{q}+\frac{1}{q'}=1$) for an adapted
duality pairing. 


\section{statement of results}
In this section, we present the main results of the paper.
Our first interest in this paper concerns Toeplitz operator from a
Bergman space to a Besov space. Recall that for  a positive Borel measure $\mu$ on $\mathcal D$ and $\nu>\frac{n}{r}-1$, the Toeplitz operator $T_\mu^\nu$ is the operator defined for any function $f$ with compact support by
\begin{equation}\label{defToeplitz}
T_\mu^\nu f(z):=\int_{\mathcal D}K_\nu(z,w)f(w)d\mu (w),
\end{equation}
where $K_\nu$ is the (weighted) Bergman kernel. Boundedness of Toeplitz operators between Bergman spaces of the unit ball was treated in \cite{PauZhao1},
the method used Carleson embedding and thus for estimations with loss, techniques of Luecking \cite{Luecking1}.
We will essentially make use of the same idea. We note that Carleson embeddings for Bergman spaces of tube domains over
symmetric were obtained by the authors in \cite{NaSeh}. We shall then prove the following result.
\btheo\label{theo:Toepbound1}
 Let $1<p\le q<\infty$, $q\ge 2$, $\alpha,\beta,\nu>\frac{n}{r}-1$ with $\nu>\frac{n}{r}-1+\frac{\beta-\frac{n}{r}+1}{q}-\frac{\alpha-\frac{n}{r}+1}{p}$. Define the numbers $\lambda=1+\frac{1}{p}-\frac{1}{q}$ and $\lambda \gamma=\nu+\frac{\alpha}{p}-\frac{\beta}{q}$. Assume that $\beta'=\beta+(\nu-\beta)q'>\frac{n}{r}-1$. Then the following assertions are equivalent.
\begin{itemize}
\item[(a)] The operator $T_\mu^\nu$ extends as a bounded operator from $A_\alpha^p(\mathcal D)$ to $\mathbb {B}_\beta^q(\mathcal D)$.
\item[(b)] There is a constant $C>0$ such that for any $\delta\in (0,1)$ and any $z\in \mathcal D,$
\Be\label{eq:Toepbound11}
\mu(B_\delta (z))\le C\Delta^{\lambda(\gamma+\frac{n}{r})}(\im z).
\Ee
\end{itemize}
\end{theorem}

We have also this estimation with loss result.
\btheo\label{theo:Toepbound2}
Let $2\le q<p<\infty$, $\alpha,\beta,\nu>\frac{n}{r}-1$ with $\nu>\frac{n}{r}-1+\frac{\beta-\frac{n}{r}+1}{q}-\frac{\alpha-\frac{n}{r}+1}{p}$. Define the numbers $\lambda=1+\frac{1}{p}-\frac{1}{q}$ and $\lambda \gamma=\nu+\frac{\alpha}{p}-\frac{\beta}{q}$. Assume $P_{\nu+m}$ is bounded on $L_\alpha^p(\mathcal D)$ for some integer $m$. Then the following assertions are equivalent.
\begin{itemize}
\item[(a)] The operator $T_\mu^\nu$ extends as a bounded operator from $A_\alpha^p(\mathcal D)$ to $\mathbb {B}_\beta^q(\mathcal D)$.
\item[(b)] For any $\delta\in (0,1)$, the function $$\mathcal {D}\ni z\mapsto \frac{\mu(B_\delta(z))}{\Delta^{\gamma+\frac{n}{r}}(\im z)}$$
belongs to $L_\gamma^{1/(1-\lambda)}(\mathcal D)$.
\end{itemize}

\etheo

\vskip .2cm
Our second interest concerns (small) Hankel operator. For $b\in A^2_\nu(\mathcal D)$, the (small) Hankel operator
with symbol $b$ is defined for $f\in \mathcal
H^\infty(\mathcal D)$
by
$$h_b^{(\nu)}(f)=h_b(f):=P_\nu(b\overline f).$$

Boundedness of Hankel operators between Bergman spaces on the unit ball was considered in \cite{BL} (see also the references therein)
where a full characterization has been obtained for estimates without loss, i.e $h_b: A_\alpha^p\mapsto A_\alpha^q$ with $1\le p\le q<\infty$.
The estimations with loss (i.e. the case $p>q$) were recently handled in \cite{PauZhao2} closing the question for the unit ball. Note that
to deal with this last case, the authors of \cite{PauZhao2} for the necessary part used an approach due to Luecking for Carleson
embeddings \cite{Luecking1}. We are
interested here in the question of the boundedness of $h_b$
from the Bergman space $A^p_\nu(\mathcal D)$ into the Besov
space $\mathbb B^q_\nu(\mathcal D)$.  For the case of estimations with loss,
we also use an adaptation of Luecking techniques for the necessary part as in \cite{PauZhao2}.
Let us denote $$q_\nu=1+\frac{\nu}{\frac{n}{r}-1},\,\,\,\textrm{and}\,\,\,\tilde
q_\nu=\frac{\nu+2\frac{n}{r}-1}{\frac{n}{r}-1}.$$ We shall the prove the following.
\begin{theorem}\label{hankel 1} 
Let $q'_\nu<p<\tilde{q_\nu}$ and $\nu>\frac{n}{r}-1$. Then the Hankel
operator $h_b$ is bounded from $A^p_\nu(\mathcal D)$ into
$\mathcal B^p_\nu(\mathcal D)$ if and only if $b=P_\nu g$ for some $g\in L^\infty (\mathbb D)$.
\end{theorem}

\begin{theorem}\label{hankel 2}
Let $1\le p<q<\infty$ and $\alpha, \beta, \nu>\frac{n}{r}-1$, $\frac{1}{p}+\frac{1}{p'}=\frac{1}{q}+\frac{1}{q'}=1$.  Define
$\beta'=\beta+(\nu-\beta)q'$; $\frac{1}{p}+\frac{1}{q'}=\frac{1}{s}<1$; $\frac{\alpha}{p}+\frac{\beta'}{q'}=\frac{\gamma}{s}$; $\frac{1}{s}+\frac{1}{s'}=1$ and $\frac{\gamma}{s}+\frac{\mu}{s'}=\nu$. Assume moreover that
$$\max\{1,\frac{\beta+\frac{n}{r}-1}{\nu},\frac{\beta-\frac{n}{r}+1}{\nu-\frac{n}{r}+1}\}<q<\infty.$$
Then the following hold.
\begin{itemize}
\item[(i)]  $h_b^{(\nu)}$ extends into a bounded operator from $A_\alpha^p(\mathcal{D})$ to $\mathbb{B}_\beta^q(\mathcal{D})$.
\item[(ii)] For some integer $m$ large enough,
$$\Delta^{m+\frac{1}{s'}(\mu+\frac{n}{r})}\Box^mb\in
L^\infty(\mathcal{D}).$$
\end{itemize}
\end{theorem}

\begin{theorem}\label{hankel 3}
 Let $2\leq q<\infty$ and $\alpha, \beta, \nu>\frac{n}{r}-1$, $\frac{1}{p}+\frac{1}{p'}=\frac{1}{q}+\frac{1}{q'}=1$.  Define
$\beta'=\beta+(\nu-\beta)q'$; $\frac{1}{p}+\frac{1}{q'}=\frac{1}{s}<1$; $\frac{\alpha}{p}+\frac{\beta'}{q'}=\frac{\gamma}{s}$; $\frac{1}{s}+\frac{1}{s'}=1$ and $\frac{\gamma}{s}+\frac{\mu}{s'}=\nu$.
Assume that
\begin{equation}
\frac{\beta-\frac{n}{r}+1}{\nu-\frac{n}{r}+1}<q
\end{equation}
and
 \begin{equation}\label{mugrand}
\frac{1}{p}(\alpha-\frac{n}{r}+1)+\frac{1}{q'}(\beta'-\frac{n}{r}+1)<\nu-\frac{n}{r}+1.
\end{equation}
Then the followig assertions hold.
\begin{itemize}
\item[(i)] If $b$ is the representative of a class in $\mathbb{B}_\mu^{s'}$, then the Hankel operator $h_b^{(\nu)}$ extends into a bounded operator from $A_\alpha^p(\mathcal D)$ to $\mathbb B^q_\beta(\mathcal D)$.
\item[(ii)] If there exists $\sigma>\frac{n}{r}-1$ such that $P_\sigma$ is bounded on $L_\alpha^p(\mathcal{D})$ and if $h_b^{(\nu)}$ extends into a bounded operator from $A_\alpha^p(\mathcal D)$ to $\mathbb B^q_\beta(\mathcal D)$, then $b$ is the representative of a class in $\mathbb{B}_\mu^{s'}$.
\end{itemize}
\end{theorem}
\vskip .2cm

Given a function $f$ in a Bergman space of some domain,
if we can write $f$ as $f=gh$ where $g$ and $h$ are in some different Bergman spaces,
then we say $f$ admits a strong factorization. C. Horowitz has proved that this is the case for
Bergman spaces of the unit disc \cite{Horowitz}. This does not longer happen in higher dimension as proved in \cite{Gowda}. 
However, it is still possible to obtain the so-called weak factorization.

For two Banach spaces of functions $A$ and $B$ defined on the same domain, the weak factored space $A\bigotimes B$ is defined as the completion of finite sums $$f=\sum_j g_j h_j, \{g_j\}\subset A, \{h_j\}\subset B$$
endowed with the following norm  $$\|f\|_{A\bigotimes B}=\inf\{\sum_j\|g_j\|_{A}\|h_j\|_{B}: f=\sum_j g_j h_j\}.$$
Given $0< p, q,s <\infty$ and $\alpha,\beta,\gamma$, whenever
the equality $A_\gamma^s(\mathcal D)=A_\alpha^p(\mathcal D)\bigotimes A_\beta^q(\mathcal D)$ holds,
we say $A_\gamma^s(\mathcal D)$ admits a weak factorization. In the case of the unit ball $\mathbb B_n$ of $\mathbb C^n$, that
weak factorization holds for Bergman spaces with small exponent (i.e. $0<s\le 1$) is a consequence of the atomic decomposition of these spaces.
This result was quite recently extended to large exponents ($s>1$) by J. Pau and R. Zhao in \cite{PauZhao2} as a consequence
of estimations with loss for the Hankel operators. As a consequence of our characterization of bounded Hankel operators, we have the following result for weighted Bergman spaces of our setting. 
\btheo\label{thm:weakfact1}
 Let $1<p, q<\infty$ and $\alpha,\beta,\nu>\frac nr-1$ so that $$\frac{1}{p}(\alpha-\frac{n}{r}+1)+\frac{1}{q}(\beta-\frac{n}{r}+1)<\nu-\frac{n}{r}+1$$ holds.
Define $\frac{1}{s}=\frac{1}{p}+\frac{1}{q}<1$ and $\frac{\gamma}{s}=\frac{\alpha}{p}+\frac{\beta}{q}$. Assume that $P_\sigma$ is bounded on $A_\alpha^p(\mathcal{D})$ for some $\sigma>\frac{n}{r}-1$, and that 
\begin{equation}\label{weakfactcond1}
\max\{q_\gamma',\frac{\gamma'+\frac{n}{r}-1}{\nu},\frac{\gamma'-\frac{n}{r}+1}{\nu-\frac{n}{r}+1}\}<s<q_\gamma\,\,\,\textrm{and}\,\,\,\max\{q_\beta',\frac{\beta+\frac{n}{r}-1}{\nu},\frac{\beta-\frac{n}{r}+1}{\nu-\frac{n}{r}+1}\}< q<q_\beta;
\end{equation}
or
\begin{equation}\label{weakfactcond2}
1<s<2\,\,\,\textrm{and}\,\,\,1<q\leq 2.
\end{equation}
Then 
 $$A_\gamma^s(\mathcal{D})=A_\alpha^p(\mathcal{D})\bigotimes A_\beta^q(\mathcal{D}).$$
\etheo
We also have the following when the weights are the same.
\btheo\label{thm:weakfact2}
Let $1<p,q<\infty$ and $\nu>\frac{n}{r}-1$. Assume that $P_\nu$ is bounded on both $L_\nu^p(\mathcal{D})$ and $L_\nu^q(\mathcal{D})$, and put $\frac{1}{s}=\frac{1}{p}+\frac{1}{q}<1$. Then 

$$A_\nu^s(\mathcal{D})=A_\nu^p(\mathcal{D})\bigotimes A_\nu^q(\mathcal{D}).$$

\etheo

As usual, given two positive quantities $A$ and $B$, the notation $A\lesssim B$ (resp. $A\gtrsim B$) means that there is an absolute
positive constant $C$ such that $A\le CB$ (resp. $A\ge CB$). When $A\lesssim B$ and $B\lesssim A$, we write $A\simeq B$ and say $A$
and $B$ are equivalent. Finally, all over the text,  $C$, $C_k$, $C_{k,j}$ will
denote positive constants depending only on the displayed
parameters but not necessarily the same at distinct
occurrences.

\section{Some useful notions and results}
\setcounter{equation}{0}
\subsection{Symmetric cones and Bergman metric}
Let $\Omega$ be an irreducible symmetric cone in the vector space $V\equiv \mathbb R^n$. Denote by $G(\Omega)$ be the group of transformations of $\mathbb R^n$ leaving invariant the cone $\Omega$, and by $e$ the identity element in
$V$. We recall that $\Omega$ induces in $V$ a structure of Euclidean Jordan algebra, in which $\overline {\Omega}=\{x^2:x\in V\}$.  
It is well known that there is a subgroup $H$ of $G(\Omega)$ that acts simply transitively on $\Omega$, that is for $x,y\in\Omega$ there is a unique $h\in H$ such that $y=hx.$ In particular, $\Omega\equiv H\cdot e$.

If we denote by $\mathbb R^n$ the group of translation by vectors in $\mathbb R^n$, then the group $G(\mathcal D)=\mathbb {R}^n\times H$ acts simply transitively on $\mathcal D$.

\vskip .2cm
Let us consider the matrix function $\{g_{jk}\}_{1\leq j,k\leq n}$ on $\mathcal D$ given by
$$g_{jk}(z)=\frac{\partial^2}{\partial z_j\partial\bar z_k}\log K(z,z)$$
where $K(w,z)=c(n/r)\Delta^{-\frac{2n}{r}}(\frac{w-\overline {z}}{i})$ is the (unweighted) Bergman kernel of $\mathcal D$. The map $z\in \mathcal {D}\mapsto \mathcal G_z$ with
$$\mathcal G_z(u,v)=\sum_{1\leq j,\,k\leq n}g_{jk}(z)u_j{\bar v}_k,\quad u=(u_1,\ldots,u_n),\,\,v=(v_1,\dots,v_n)\in \mathbb C^n$$
defines a Hermitian metric on $\mathbb C^n,$ called the Bergman metric. The Bergman length of a smooth path $\gamma:[0,1]\to \mathcal D$ is given by
$$l(\gamma)=\int_0^1\left(\mathcal G_{\gamma(t)}(\dot\gamma(t),\dot\gamma(t))\right)^\frac{1}{2}dt$$
and the Bergman distance $d(z_1,z_2)$ between two points $z_1, z_2\in \mathcal D$ is defined by
$$d(z_1,z_2)=\inf_\gamma l(\gamma)$$
where the infimum is taken over all smooth paths $\gamma:[0,1]\to \mathcal{D}$ such that $\gamma(0)=z_1$ and $\gamma(1)=z_2.$

For $\delta>0$, we denote by $$B_\delta(z)=\{w\in \mathcal {D}: d(z,w)<\delta\}$$ the Bergman ball centered at $z$ with radius $\delta.$

We refer to \cite[Theorem 5.4]{BBGNPR} for the following.
\blem\label{lem:covering}
Given $\delta\in (0, 1)$, there exists a sequence $\{z_j\}$ of points  of $\mathcal D$ called
$\delta$-lattice such that, if $B_j=B_\delta(z_j)$ and $B_j'=B_{\frac{\delta}{2}}(z_j)$,
then
\begin{itemize}
\item[(i)] the balls $B_j'$ are pairwise disjoint;
\item[(ii)] the balls $B_j$ cover $\mathcal D$ with finite overlapping, i.e. there is an integer N (depending only on $\mathcal{D}$) such
that each point of $\mathcal D$ belongs to at most N of these balls.
\end{itemize}
\elem

The above balls have the following
properties:$$\int_{B_j}dV_\nu(z)\approx \int_{B'_j}dV_\nu(z)\approx
C_{\delta}\Delta^{\nu+n/r}(\im z_j).$$ We recall that the
measure $d\lambda(z)=\Delta^{-2n/r}(\im z)dV(z)$ is an
invariant measure on $\mathcal D$ under the actions of $G(\mathcal D)=\mathbb {R}^n\times H$.

\vskip .2cm
Let us denote by $l_{\nu}^{p}$, the space of complex sequences $\beta=\{\beta_j\}_{j\in \N}$ such that
$$||\beta||_{l_\nu^p}^{p}=\sum_{j}
|\beta_{j}|^{p}\Delta^{\nu+\frac{n}{r}}(\im z_j)<\infty,$$
where $\{z_j\}_{j\in \N}$ is a $\delta$-lattice.
\vskip .2cm
The following is quite easy to check.
\blem\label{5.1}\cite[Lemma
2.11.1 ]{G} Suppose $1 \le p < \infty$, $\nu $ and $\mu $
are reals. Then, the dual space $(l_{\nu}^{p})^{*}$ of the space
$l_{\nu}^{p}$ identifies with $l_{\nu + (\mu - \nu)p'}^{p'}$
under the sum pairing $$\langle\eta,\beta\rangle_{\mu}=
\sum_{j}\eta_{j}\overline {\beta_{j}}\Delta^{\mu +
\frac{n}{r}}(\im z_j),$$ where $\eta = \{\eta_{j}\}$ belongs to
$l_{\nu}^{p}$ and $\beta =\{\beta_{j}\}$ belongs to $l_{\nu + (\mu
- \nu)p'}^{p'}$ with $\frac{1}{p} +\frac{1}{p'}=1$.
\elem

Note that in the text, the space $l^p$, is the usual sequence space. 

The following result known as the sampling theorem.
\blem\label{5.3}\cite[Theorem 5.6]{BBGNPR}
Let $\{z_{j}\}_{j\in \N}$ be a $\da$-lattice in $\mathcal D$,
$\da \in (0,1)$ with $z_{j}= x_{j} + iy_{j}$. The following
assertions hold.
\begin{itemize}
\item[(1)] There is a positive
constant $C_{\da}$ such that every $f\in A_{\nu}^{p}(\mathcal D)$ satisfies
$$ ||\{f(z_{j})\}||_{l_\nu^p} \le C_{\da} ||f||_{p,\nu}. $$
\item[(2)] Conversely, if $\da$ is small enough, there is a
positive constant $C_{\da}$ such that every $f\in A_{\nu}^{p}(\mathcal D)$
satisfies $$||f||_{p,\nu} \le
C_{\da}||\{f(z_{j})\}||_{l_\nu^p}.$$
\end{itemize}
\elem
We will need the following consequence of the mean value theorem (see \cite{BBGNPR})
\blem\label{lem:meanvalue}
There exists a constant $C>0$ such that for any $f\in \mathcal {H}(\mathcal D)$ and $\delta\in (0,1]$, the following holds
\begin{equation}\label{eq:meanvalue}
|f(z)|^p\le C\delta^{-n}\int_{B_\delta(z)}|f(\zeta)|^p\frac{dV(\zeta)}{\Delta^{2n/r}(\im \zeta)}.
\end{equation}
\elem
We finish this subsection with the following pointwise estimate of functions in Bergman spaces.
\blem\label{lem:pointwiseestimberg}\cite[Proposition 3.5]{BBGNPR}
Let $1\leq p<\infty$, and $\nu>\frac{n}{r}-1$. Then there is a constant $C>0$ such that for any $f\in A_\nu^p(\mathcal D)$ the following
pointwise estimate holds:
\begin{equation}\label{eq:pointwiseestimberg}|f(z)|\le C\Delta^{-\frac{1}{p}(\nu+\frac{n}{r})}(\im z)\|f\|_{A_\nu^p},\,\,\,\textrm{for all}\,\,\, z\in \mathcal {D}.
\end{equation}
\elem
\subsection{Isomorphism of Besov spaces}
Let us denote by  $\mathbb B^{p,(k)}_\nu(\mathcal D)$
(respectively $\mathbb B^{\infty,(k)}$) the space obtained
by replacing $k_0$ (respectively $m_*$) by $k\ge k_0$
(respectively $k\ge m_*$) in the definition of $\mathbb
B^p_\nu(\mathcal D)$ (respectively $\mathbb
B^{\infty,(m)}$).  Then it is not hard to prove the
following result (see \cite{BBGRS} for details).

\bprop The
natural projection 

\begin{tabular}{lp{2in}}
 & \\
$\mathbb B^{p,(k)}_\nu(\mathcal D)\,\,\,
(\textrm{respectively}\,\,\, \mathbb B^{\infty,(k)})
\longrightarrow \mathbb B^{p,(m)}_\mu(\mathcal D)\,\,\,
(\textrm{respectively}\,\,\, \mathbb B^{\infty,(m)})$\\
 & \\
$F+\mathcal N_k  \longmapsto  F+\mathcal N_m$ is an
isomorphism of Banach spaces.
\end{tabular}
\eprop
We will heavily make use of the above
proposition.

\subsection{Bergman kernel and reproducing formulas}
The (weighted) Bergman projection $P_\nu$ is defined by
$$P_\nu f(z)=\int_{\mathcal D}K_\nu(z, w) f(w)
dV_\nu(w),
$$
where
 $K_\nu(z, w)=
 c_\nu\,\Delta^{-(\nu+\frac{n}{r})}((z-\overline {w})/i)$
is the Bergman kernel, i.e the reproducing kernel of  $A^2_\nu$ (see
\cite{FK}). Here, we use the notation
$dV_\nu(w):=\Delta^{\nu-\frac{n}{r}}(v) du\,dv$, where
$w=u+iv$ is an element of $\mathcal D$.

We recall that the Box operator $\Box$ acts on the Bergman
kernel in the following
way:\begin{equation}\label{boxofkernel}\Box^m{K_\nu(z,.)=C_{\nu,m}}K_{\nu+m}(z,.)\end{equation}
(see \cite{BBPR}).

We will need the following integration by parts formula which
follows from the density of the intersection of two Bergman
spaces in each of them (see \cite{BBGRS}, \cite{BBPR}). For
$\nu>\frac{n}{r}-1$, $1\leq p\leq \infty$ and $f\in
A^p_\nu$, $g\in A^{p'}_\nu$, we have the formula
\begin{equation}\label{parts}
\int_{\mathcal D} f(z)\overline g(z) dV_\nu(z)=c_{\nu,
m}\int_{\mathcal D} f(z)\overline {\Box^m g(z)} \Delta^m(\im z)
dV_\nu(z).
\end{equation}
 Using an
adapted version of the above formula for the Box of
functions, we can prove the following (see \cite{BBGRS} for
details).

\begin{proposition}\label{Repr1}
Let $\nu>\frac nr-1$ and $1\leq p\leq \infty$. For all $f\in
A^p_\nu$ we have the formula
\begin{equation}\label{repr}
  \Box^\ell f(z)=c\int_{\cD} K_{\nu+\ell} (z, w)\Box^m
  f(w)\Delta^m(\im w)dV_\nu(w)
\end{equation}
for $m\geq 0$ and $\ell$ large enough so that $K_{\nu+\ell} (z,
\cdot)$ is in $L^{p'}_\nu$. In particular, when $1\leq p< \tilde
p_\nu$, the formula is valid with $\ell=0$.
\end{proposition}

\bcor \label{Repr2}Let $1\le p<\tilde p_\nu$ and
$\nu>\frac{n}{r}-1$. Then, any $f\in A^p_\nu$ satisfies the
formula
\begin{equation}\label{repr2}
   f(z)=\int_{\mathcal D} K_{\nu} (z, w)f(w)dV_\nu (w).\end{equation}
\ecor

The following is Proposition 2.19 in \cite{BBGRS}.

\begin{proposition}\label{Repr3}
\label{pro2} Let $\mu,\nu, \alpha \in\mathbb {R}$ and $1\leq p<\infty$
satisfying \[\nu + \alpha > \frac{n}{r}-1,\quad
\nu p-\mu>(p-1)(\frac{n}{r}-1)\mand\mu + \al p >
(p-1)(\frac{n}{r}-1)-\frac{n}{r}.\] Then, the function $z\mapsto
\Delta^{\nu-\mu}(\im z)K_{\nu+\alpha}(z, i\be)\in
L^{p'}_\mu(\mathcal D)$, and for all holomorphic function $f$ such
that the function $z\mapsto \Delta^{\alpha}(\im z)f(z)\in L^p_\mu,$
we have \Be f(z)=\int_{\mathcal D} K_{\nu+\al} (z,
w)\,f(w)\,\Delta^{\al}(\im w)\,dV_\nu(w). \label{rep3} \Ee
\end{proposition}


 The next lemma gives integrability properties of  the
determinants and Bergman kernels.

\begin{lemma}\label{int-compl} Let $\alpha, \beta,\nu$ be real.
Then
\begin{itemize} \item[1)] for $y\in\Omega$, the integral
$$J_{\alpha}(y)=\int_{\mathbb{R}^{n}}\left|\Delta^{-\alpha}((x+iy)/i)\right|dx
$$ converges if and only if $\alpha > \frac{2n}{r} -1.$ In this case,
$J_{\alpha}(y)=C_{\alpha}\Delta^{-\alpha +\frac{n}{r}}(y)$,
where $C_{\alpha}$ is a constant depending only on
$\alpha$.
\item[2)]The function
$f(z)=\Delta^{-\alpha}(\frac{z+it}{i})$, with $t \in
\Omega$, belongs to $A_{\nu}^{p}$ if and only if
$$\nu>\frac{n}{r}-1\mand \alpha > \frac{1}{p}(\nu+\frac{2n}{r} -1).$$ In this
case,$$||f||_{A_{\nu}^{p}}=C_{\alpha,p}\Delta^{-\alpha
+(\nu+\frac {n}{r})\frac{1}{p}}(t).$$ \end{itemize}
\end{lemma}
\Proof See \cite{BBGNPR}.\ProofEnd

\medskip

\subsection{Integral operators and duality}
Let us now consider the following integral operators
$T=T_{\mu,\alpha}$ and $T^+=T^+_{\mu,\alpha}$, defined by
\begin{equation}\label{eq:T}
  Tf(z)=\Delta^{\alpha}(\im z)\int_{\mathcal D}K_{\mu+\alpha}(z,w)f(w)dV_\mu(w),
\end{equation}
 and
 \begin{equation}\label{eq:T+}
T^+f(z)=\Delta^{\alpha}(\im z)\int_{\mathcal
D}|K_{\mu+\al}(z,w)|f(w)dV_\mu(w),
\end{equation}
 provided these integrals make sense.
 Observe that $P_{\mu}=T_{\mu,0}$.

The following result is in \cite{sehba}.
 \begin{lemma}\label{lem:bergtype} Let $\alpha,\mu,\nu\in\mathbb R$ and $1\le
 p<\infty$.
 Then the following conditions are
 equivalent:
\begin{itemize}
 \item[(a)]
The operator $T^+_{\mu,\alpha}$ is  well defined and
bounded on $L_{\nu}^{p}(\mathcal D)$. \item[(b)] The
parameters satisfy $\mu+\alpha>\frac{n}{r}-1$ and the
inequalities
\[
\mu p-\nu>(\frac{n}{r}-1)\max \{1, p-1\}, \quad \alpha
p+\nu\,>(\frac{n}{r}-1)\max \{1, p-1\}.\]
 \end{itemize} \end{lemma}
 Also, from \cite{sehba} we have the following.
 \begin{lemma}\label{lem:bergtypeinfty} Let $\alpha,\mu\in\mathbb R$. Then both $T$ and $T^+$ are bounded on $L^\infty (\mathcal D)$
 if and only $\alpha>\frac{n}{r}-1$ and $\mu>\frac{n}{r}-1$.
 \end{lemma}
We will also need the following result.
\begin{lemma}[B\'ekoll\'e 1986 \cite{Be}] \label{A1}Let
$\mu > \frac{n}{r}-1$ and $m$ an integer such that $m>\frac
nr -1$. Then the dual space
 $(A_{\mu}^{1})^*$ identifies with the Bloch space $\mathbb B^\infty$ under the integral pairing
\Be\langle f,g\rangle_{\mu,m}=\int_{\mathcal
D}f(z)\overline {\Delta^m\Box^m g(z)}dV_\mu(z),\quad f\in
A_{\mu}^{1},\;\; g\in\mathbb B^\infty\label{b1}.\Ee
\end{lemma}

The following duality with change of weights can be proved in the same way as the
previous one.
\begin{lemma}\label{lem:ApBp0}
Let $\nu,\mu>\frac{n}{r}-1$ and $1<p\leq2$. Then,
$(A_{\nu}^{p})^{*}$ identifies with $\mathbb B_{\nu+(\mu-\nu)p'}^{p'}$
under the integral pairing \Be \langle
f,g\rangle_{\nu,m}=\int_{\mathcal D}f(z)\overline
{\Delta^{m}\Box^{m}g(z)}
\,dV_\mu(z),\quad f\in A_{\nu}^{p},\quad g\in \mathbb
B_{\nu+(\mu-\nu)p'}^{p'},\label{b3}\Ee for any integer $m\geq
k_0(p',\nu+(\mu-\nu)p')$.
\end{lemma}

We recall the following notations:
$$\tilde q_{\nu,p}=\frac{\nu + \frac{n}{r}-1}{(\frac{n}{rp'}-1)_+},\,\,\,\,\,\,\,\,\,
q_{\nu,p}=\min\{p,p'\}q_\nu,\,\,\,\,\,\textrm{and}\,\,\,\,\,\,
q_\nu =1+ \frac{\nu}{\frac{n}{r}-1}
$$ with $\tilde q_{\nu,p}=\infty$, if $n/r\le p'$. It is clear that $1<q_\nu<q_{\nu,p}<\tilde
q_{\nu,p}$. By density of the intersection $A_{\nu}^{p}\cap
A_{\mu}^2$ in $A_{\nu}^{p}$, we have the following reproducing
formula for all $\a>\frac{n}{r}-1$ and $f\in A_{\nu}^{p}$ with
 $1\le p<\tilde q_{\nu,p}$. \Be\label{rep}
f(z)=\int_{\mathcal D}K_\a(z,w)f(w)\Da^{\a-\frac{n}{r}}(\Im
w)dV(w),\,\,\, z\in \mathcal D.\Ee  The following theorem obtained in \cite{sehba} characterizes
the topological dual space of the Bergman space $A_{\nu}^{p}$
for some values of $p$ and $\nu$ for which the Bergman
projection is not necessarily bounded.
\btheo\label{theo:dualitychange} Let
$\nu>\frac{n}{r}-1$ be real, $1<p<q_\nu$. If
$\mu$ is a sufficiently large real number so that
$\mu>\frac{n}{r}-1$ and $1 < p' <q_{\mu}$, then the topological
dual space $(A_{\nu}^{p})^*$ of the Bergman space
$A_{\nu}^{p}$ identifies with $A_{\mu}^{p'}$ under the
integral pairing
$$\langle f,g\rangle_\a=\int_{\mathcal D}f(w)\overline
{g(w)}\Da^{\a -\frac{n}{r}}(\Im w)dV(w) $$ where
$\a=\frac{\nu}{p}+\frac{\mu}{p'}$, $\frac{1}{p} +\frac{1}{p'}=1$.
\etheo
\subsection{Atomic decomposition}
The following general atomic decomposition of functions in weighted Bergman space is from \cite{NaSeh}.
\btheo\label{theo:atomdecompo} Let $1 < p <\infty$ and let $\mu,\,\,\nu >\frac{n}{r}-1$ satisfying $\nu +
(\mu - \nu)p' > \frac{n}{r}-1.$ Assume that
the operator $P_{\mu}$ is bounded on $L_{\nu}^{p}(\mathcal D)$ and let $\{z_{j}\}_{j\in \N}$
be a $\da$-lattice in $\mathcal D$. Then the following assertions hold.
\begin{itemize}
\item[(i)] For every complex sequence $\{\la_{j}\}_{j\in \N}$ in
$l_{\nu}^{p}$, the series  $\sum_{j} {\lambda_{j}K_{\mu}(z,
z_{j})\Delta^{\mu + \frac{n}{r}}(y_{j})}$ is convergent in
$A_{\nu}^{p}(\mathcal D)$. Moreover, its sum $f$ satisfies the inequality
$$||f||_{p,\nu}\leq C_{\delta}||\{\la_{j}\}||_{l_{\nu}^{p}},$$ where
$C_{\delta}$ is a positive constant.
\item[(ii)] For $\da$ small
enough, every function $f\in A_{\nu}^{p}(\mathcal D)$ may be written as
$$ f(z) = \sum_{j} {\lambda_{j}K_{\mu}(z,
z_{j})\Delta^{\mu + \frac{n}{r}}(y_{j})}$$ with
\begin{equation}\label{eq:reverseineqatomdecomp}
||\{\la_{j}\}||_{l_{\nu}^{p}}\le
C_{\da}||f||_{p,\nu}\Ee where $C_{\da}$ is a positive
constant.
\end{itemize}
\etheo
\subsection{Properties of $ \mathbb {B}^p_\nu(\mathcal D)$, image of the Bergman operator.}
The notion of the Bergman projection has been extended in \cite{BBGRS} allowing to see $\mathbb {B}^p_\nu(\mathcal D)$ as a projection of $L^p_\nu(\mathcal D)$
in some sense. More precisely, for $1\le p<\infty$ and $\nu\in \mathbb R$ and $m$
large enough, the Bergman projection $P_\nu^{(m)}(f)$ of $f\in L_\nu^p(\mathcal D)$ is defined as the equivalent
class (in $\mathcal {H}_m$) of all holomorphic solutions of
\begin{equation}\label{eq:extendedBergproj}\Box ^mF=C_{\nu,m}\int_{\mathcal D}K_{\nu+m}(\cdot,w)f(w)dV_\nu(w).\end{equation}
The constant $C_{\nu,m}$ is as in (\ref{boxofkernel}). One easily observes that for any $F$ in
the equivalent class $P_\nu^{(m)}(f)$, $$\Delta^m\Box ^mF=C_{\nu,m}T_{\nu,m}f$$ and consequently, $P_\nu^{(m)}$ is well defined
and bounded from $L_\nu^p(\mathcal D)$ to $\mathbb {B}^p_\nu(\mathcal D)$ if and only if $T_{\nu,m}$ is bounded in $L_\nu^p(\mathcal D)$.
It follows in particular with the help of Proposition \ref{Repr1} that every $F\in \mathbb {B}^{p,(m)}_\nu(\mathcal D)$
satisfies
\begin{equation}\label{eq:extendedBergproj1}\Box ^mF=C_{\nu,m}\int_{\mathcal D}K_{\nu+m}(\cdot,w)\Delta^m(\im w)\Box ^mF(w)dV_\nu(w)\end{equation}
whenever $m$ is sufficiently large.

\subsection{Kor\'anyi's Lemma and averaging of a measure}
Let $\mu$ be a positive measure on $\mathcal D$. For $z\in \mathcal D$ and $\delta\in (0,1)$, we define the average of the positive measure $\mu$ at $z$ by $$\hat {\mu}_\delta(z)=\frac{\mu(B_\delta(z))}{V_\nu(B_\delta(z))}.$$
The following is also needed here.
\blem\label{lem:integraldiscretizationAverBer}
Let $1\le p\le \infty$, $\nu\in \mathbb R$, $\beta, \delta\in (0,1)$. For $z\in \mathcal D$ we define the average of the positive measure $\mu$ at $z$ by $$\hat {\mu}_\delta(z)=\frac{\mu(B_\delta(z))}{V_\nu(B_\delta(z))}.$$ Let $\{z_j\}_{j\in \mathbb N}$ be a $\delta$-lattice in $\mathcal D$. Then the following assertions are equivalent.
\begin{itemize}
\item[(i)] $\hat {\mu}_\beta\in L_\nu^p(\mathcal D)$.
\item[(ii)] $\{\hat {\mu}_\delta(z_j)\}_{j\in \mathbb N}\in l_\nu^p$.
\end{itemize}
\elem
\vskip .2cm
For $\nu>\frac{n}{r}-1$ and $w\in \mathcal D$,  the normalized reproducing kernel is
\begin{equation}\label{eq:normalrepkern} k_\nu(\cdot,w)=\frac{K_\nu(\cdot,w)}{\|K_\nu(\cdot,w)\|_{2,\nu}}=\Da^{-\nu-\frac nr}\left(\frac{\cdot-\bar w}{i}\right)\Da^{\frac{1}{2}(\nu+\frac nr)}(\Im w).
\end{equation}
We have the following result known as Kor\'anyi's lemma.
\blem\label{lem:Koranyi}\cite[Theorem 1.1]{BIN}
For every $\delta>0,$ there is a constant $C_{\delta}>0$ such that
$$\left|\frac{K(\zeta,z)}{K(\zeta,w)}-1\right|\leq C_{\delta}d(z,w)$$
for all $\zeta,z,w\in \mathcal D,$ with $d(z,w)\leq \delta.$
\elem
The following corollary is a straightforward consequence of \cite[Corollary 3.4]{BBGNPR} and Lemma \ref{lem:Koranyi}.
\bcor\label{kor}
Let $\nu>\frac nr-1,\,\,\da>0$ and $z,w\in \mathcal D.$  
There is a positive constant $C_\da$ such that for all $z\in B_\da(w),$
$$V_\nu(B_\da(w))|k_\nu(z,w)|^2\leq C_\da.$$

If $\da$ is sufficiently small, then there is $C>0$ such that for all $z\in B_\da(w),$
$$V_\nu(B_\da(w))|k_\nu(z,w)|^2\geq (1-C\da).$$
\ecor
\subsection{Carleson embeddings}
We start by recalling the following from \cite[Theorem 3.5]{NaSeh}.
\btheo\label{theo:Carlembed1}
Suppose that $\mu$ is a positive measure on $\mathcal D$, $1\le p<\infty$, $\nu>\frac{n}{r}-1$ and $\lambda\ge 1$. Then there
exists a constant $C>0$ such that
\begin{equation}\label{eq:Carlembed11}
\int_{\mathcal D}|f(z)|^{p\lambda}d\mu (z)\le C\|f\|_{p,\nu}^{p\lambda},\,\,\,f\in A_\nu^p(\mathcal D)
\end{equation}
if and only if
\begin{equation}\label{eq:Carlembed12}
\mu(B_\delta(z))\le C'\Delta^{\lambda(\nu+\frac{n}{r})}(\im z),\,\,\,z\in \mathcal D,
\end{equation}
for some $C'>0$ independent of $z$, and for some $\delta\in (0,1)$.
\etheo
\vskip .2cm
Note that condition (\ref{eq:Carlembed12}) is still sufficient for (\ref{eq:Carlembed11}) to holds in the case $0<p\leq 1$. The following embedding with loss was also obtained in \cite[Theorem 3.8]{NaSeh}.
\btheo\label{theo:Carlembed2}
Suppose that $\mu$ is a positive measure on $\mathcal D$,  $\nu>\frac{n}{r}-1$ and $0<\lambda<1$. Then the following assertions
hold.
\begin{itemize}
\item[(i)] Let $0< p <\infty.$ If the function
$$\mathcal {D}\ni z\mapsto \frac{\mu(B_\delta(z))}{\Delta^{\nu+\frac{n}{r}}(\im z)}$$
belongs to $L_\nu^{s}(\mathcal D)$, $s=\frac{1}{1-\lambda}$ for some $\delta\in (0,1)$, then
there
exists a constant $C>0$ such that
\begin{equation}\label{eq:Carlembed21}
\int_{\mathcal D}|f(z)|^{p\lambda}d\mu (z)\le C\|f\|_{p,\nu}^{p\lambda},\,\,\,f\in A_\nu^p(\mathcal D).
\end{equation}
\item[(ii)] Let $1\le p <\infty.$ If the Bergman projection $P_\alpha$ is bounded on $L_\nu^p(\mathcal D)$ and (\ref{eq:Carlembed21}) holds, then
the function
$$\mathcal {D}\ni z\mapsto \frac{\mu(B_\delta(z))}{\Delta^{\nu+\frac{n}{r}}(\im z)}$$
belongs to $L_\nu^{s}(\mathcal D)$, $s=\frac{1}{1-\lambda}$ for some $\delta\in (0,1)$.
\end{itemize}
\etheo

Any measure satisfying (\ref{eq:Carlembed11}) or (\ref{eq:Carlembed21}) is called $(\lambda, \nu)$-Carleson measure.

\section{Toeplitz operators from $A_\alpha^p(\mathcal D)$ to $\mathbb {B}_\beta ^q(\mathcal D).$}
In this section we provide criteria for the boundedness of the Toeplitz operators from a weighted Bergman space to an analytic Besov space. In particular, we prove here Theorem \ref{theo:Toepbound1} and Theorem \ref{theo:Toepbound2}.

\begin{proof}[Proof of Theorem \ref{theo:Toepbound1}]
Let us start with the sufficiency part. Put $\beta'=\beta+(\nu-\beta)q'>\frac{n}{r}-1$. We observe that
by Lemma \ref{lem:ApBp0}, $A_{\beta'}^{q'}(\mathcal D)$ is the dual space of $\mathbb {B}_\beta^q$ under
the duality pairing $$\langle f,g\rangle_{\nu,m}=\int_{\mathcal D}f(z)\Delta^m(\im z)\overline {\Box^mg(z)}dV_\nu(z)$$ for $m$ large enough.
\vskip .1cm
We also observe the condition $\nu>\frac{n}{r}-1+\frac{\beta-\frac{n}{r}+1}{q}-\frac{\alpha-\frac{n}{r}+1}{p}$ provides
that $\gamma>\frac{n}{r}-1$, and that as $\lambda\ge 1$, so by Theorem \ref{theo:Carlembed1}, for every $t\geq 1,$ there exists a constant $C>0$  such that
$$\int_{\mathcal D}|h(z)|^{t\lambda}d\mu(z)\leq C\|h\|_{t,\gamma}^{t\lambda}$$ for all $h\in A_\gamma^t({\mathcal D}).$ Taking $t\lambda=1,$
this is, in particular, equivalent to saying that
there is a constant $C>0$ such that for any $h\in L_\gamma^{\frac{1}{\lambda}}(\mathcal D)$,
\begin{equation}\label{lambda}\int_{\mathcal D}|h(z)|d\mu(z)\le C\|h\|_{\frac{1}{\lambda},\gamma}.\end{equation}
Let $f\in A_\alpha^p(\mathcal D)$ and $g\in A_{\beta'}^{q'}(\mathcal D)$. Using Fubini's Theorem and reproducing formula for weighted Bergman space, we obtain that
\begin{equation}\label{lambda 1}
 \langle g,T_\mu^\nu f\rangle_{\nu,m}=\int_{\mathcal D}\overline {f(z)}g(z)d\mu(z).
\end{equation}

Also $fg\in L_\gamma^{\frac{1}{\lambda}}(\mathcal D).$ As a matter of fact, by H\"oler's inequality, we have
\begin{eqnarray*}
\|fg\|_{\frac{1}{\lambda},\gamma}&=&\left(\int_{\mathcal D}|f(z)|^{1/\lambda}|g(z)|^{1/\lambda}\Delta^{\frac{1}{\lambda}(\nu+\frac{\alpha}{p}-\frac{\beta}{q})}(\im z)\frac{dV(z)}{\Delta^{n/r}(\im z)}\right)^\la\\
&\leq&\left(\int_{\mathcal D}|f(z)|^pdV_\alpha(z)\right)^{1/p}\left(\int_{\mathcal D}|g(z)|^{q'}\Delta^{q'(\nu-\frac{\beta}{q})-\frac{n}{r}}(\im z)dV(z)\right)^{1/q'}\\
&\leq& \|f\|_{p,\alpha}\|g\|_{q',\beta'}.
\end{eqnarray*}

Therefore, from (\ref{lambda 1}) and (\ref{lambda}),  we get
\Beas
|\langle g,T_\mu^\nu f\rangle_{\nu,m}| &\le& \int_{\mathcal D}|f(z)g(z)|d\mu(z)\leq C\|fg\|_{\frac{1}{\lambda},\gamma}\leq \|f\|_{p,\alpha}\|g\|_{q',\beta'}.\\ .
\Eeas
Thus, $T_\mu^\nu$ is bounded from $A_\alpha^p(\mathcal D)$ to $\mathbb {B}_\beta^q(\mathcal D)$ whenever $(\textrm{b})$ holds.
\vskip .1cm
Next, suppose that $T_\mu^\nu$ extends as a bounded operator from $A_\alpha^p(\mathcal D)$ to $\mathbb {B}_\beta^q(\mathcal D)$.
For $a\in \mathcal D$ given, we define
$$f_a(z)=K_{\nu+m}(z,a)$$ where $m$ is an integer large enough.
From Lemma \ref{int-compl} we obtain that $f_a$ belongs to
$A_\alpha^p(\mathcal D)$ with \Be\label{eq:normtesttoep1}\|f_a\|_{p,\alpha}=C_{p,m,\nu,\alpha}\Delta^{-(m+\nu+\frac{n}{r})+(\alpha+\frac{n}{r})\frac{1}{p}}(\im a).\Ee
We have
\Beas
T_\mu^\nu f_a(z) &=& \int_{\mathcal D}f_a(w)K_\nu(z,w)d\mu(w)\\ &=& \int_{\mathcal D}K_{\nu+m}(w,a)K_\nu(z,w)d\mu(w)
\Eeas
hence
$$\Box^mT_\mu^\nu f_z(z)=C_m\int_{\mathcal D}|K_{\nu+m}(z,w)|^2d\mu(z).$$
It follows in particular, using Corollary \ref{kor}, that for any $\delta\in (0,1)$ and any $z\in \mathcal D$, 
\Be\label{eq:estimmutest}\Box^mT_\mu^\nu f_z(z)\ge C_m\frac{\mu(B_\delta(z))}{\Delta^{2(m+\nu+\frac{n}{r})}(\im z)}.\Ee
On the other hand, we have from our hypothesis that $\Box^mT_\mu^\nu f_a\in A_{\beta+mq}^q(\mathcal D)$, thus from the pointwise estimate (\ref{eq:pointwiseestimberg}) and (\ref{eq:normtesttoep1}) that
\Beas
|\Box^mT_\mu^\nu f_z(z)| &\le& C\Delta^{-(\beta+mq+\frac{n}{r})\frac{1}{q}}(\im z)\|\Box^mT_\mu^\nu f_z\|_{q,\beta+mq}\\ &\le& C\Delta^{-(\beta+mq+\frac{n}{r})\frac{1}{q}}(\im z)\|T_\mu^\nu\|\|f_z\|_{p,\alpha}\\ &\le& C\|T_\mu^\nu\|\Delta^{-(\nu+2m+\frac{n}{r})+\frac{\alpha+\frac{n}{r}}{p}-\frac{\beta+\frac{n}{r}}{q}}(\im z).
\Eeas
Combining the latter with (\ref{eq:estimmutest}) we obtain that
$$\frac{\mu(B_\delta(z))}{\Delta^{2(m+\nu+\frac{n}{r})}(\im z)}\le C\|T_\mu^\nu\|\Delta^{-(\nu+2m+\frac{n}{r})+\frac{\alpha+\frac{n}{r}}{p}-\frac{\beta+\frac{n}{r}}{q}}(\im z)$$ that is
$$\mu(B_\delta(z))\le C\|T_\mu^\nu\|\Delta^{\nu+\frac{n}{r}+\frac{\alpha+\frac{n}{r}}{p}-\frac{\beta+\frac{n}{r}}{q}}(\im z)$$
or equivalently
$$\mu(B_\delta(z))\le C\|T_\mu^\nu\|\Delta^{\lambda(\gamma+\frac{n}{r})}(\im z).$$
The proof is complete.
\end{proof}
The following can be proved the same using the duality result of Theorem \ref{theo:dualitychange}.
\btheo\label{theo:Toepbound3}
Let $1<p\le q<q_\beta$, $\alpha,\beta,\nu>\frac{n}{r}-1$ with $\nu>\frac{n}{r}-1+\frac{\beta-\frac{n}{r}+1}{q}-\frac{\alpha-\frac{n}{r}+1}{p}$. Define the numbers $\lambda=1+\frac{1}{p}-\frac{1}{q}$ and $\lambda \gamma=\nu+\frac{\alpha}{p}-\frac{\beta}{q}$. Assume that $1<q'<q_{\beta'}$ where $\beta'=\beta+(\nu-\beta)q'>\frac{n}{r}-1$. Then the following assertions are equivalent.
\begin{itemize}
\item[(a)] The operator $T_\mu^\nu$ extends as a bounded operator from $A_\alpha^p(\mathcal D)$ to $\mathbb {B}_\beta^q(\mathcal D)$.
\item[(b)] There is a constant $C>0$ such that for any $\delta\in (0,1)$ and any $z\in \mathcal D$.
\Be\label{eq:Toepbound31}
\mu(B_\delta (z))\le C\Delta^{\lambda(\gamma+\frac{n}{r})}(\im z)
\Ee
\end{itemize}
\etheo
We now prove the case of estimations with loss.

\begin{proof}[Proof of Theorem \ref{theo:Toepbound2}]
That $(\textrm{b})\Rightarrow (\textrm{a})$ follows as in the proof of Theorem \ref{theo:Toepbound1} using Theorem \ref{theo:Carlembed2}. Let us prove that $(\textrm{a})\Rightarrow (\textrm{b})$.
We start by recalling that the Rademacher functions $r_j$ are given by $r_j(t)=r_0(2^jt)$, for $j\ge 1$ and $r_0$ is defined as follows
$$
r_0(t)=\left\{ \begin{matrix} 1 &\text{if }& 0\le t-[t]<1/2\\
      -1 & \text{ if } & 1/2\le t-[t]<1
                                  \end{matrix} \right.
$$
$[t]$ is the smallest integer $k$ such $k\le t<k+1$. We have the following.
\blem[Kinchine's inequality]\label{lem:Kinchine}
For $0 < p <\infty$ there exist constants $0 <L_p\le M_p <\infty$ such that, for all natural numbers $m$ and all complex
numbers $c_1, c_2,\cdots, c_m,$ we have
$$L_p\left(\sum_{j=1}^m|c_j|^2\right)^{p/2}\le \int_0^1\left|\sum_{j=1}^mc_jr_j(t)\right|^pdt\le M_p\left(\sum_{j=1}^m|c_j|^2\right)^{p/2}.$$
\elem
\vskip .1cm
Coming back to the proof, since $P_{\nu+m}$ is bounded on $L_\alpha^p(\mathcal D)$, it follows from Theorem \ref{theo:atomdecompo}
that given a $\delta$-lattice $\{z_j\}_{j\in \mathbb N_0}$, $\delta\in (0,1)$,
for every complex sequence $\{\lambda_j\}\in l_{\alpha}^p$, the series $$\sum_j\lambda_jK_{\nu+m} (z,z_j)\Delta^{\nu+m+n/r}(\im z_j)$$
is convergent in $A_\alpha^p(\mathcal D)$ and its sum $f$ satisfies
$$\|f\|_{p,\alpha}\le C\left(\sum_j|\lambda_j|^p\Delta^{\alpha+n/r}(\im z_j)\right)^{1/p}.$$
Thus if $r_j(t)$ is a sequence of Rademacher functions and $\{\lambda_j\}\in l_{\alpha}^p$,
$$f_t(z)=\sum_j\lambda_jr_j(t)K_{\nu+m} (z,z_j)\Delta^{\nu+m+n/r}(\im z_j)$$ also converges in $A_{\alpha}^p(\mathcal D)$ with $\|f_t\|_{p,\alpha}\le C\|\{\lambda_j\}\|_{l_\alpha^p}$.
\vskip .1cm
As $T_\mu^\nu$ extends as a bounded operator from $A_\alpha^p(\mathcal D)$ to $\mathbb {B}_\beta^q(\mathcal D)$,
we obtain for $m$ integer large enough (actually, $P_{\nu+m}$ is bounded on $L_\alpha^p(\mathcal D)$),
\Beas
\|\Box^mT_\mu^\nu f_t\|_{q,\beta+mq}^q &=& \int_{\mathcal D}\left|\sum_j\lambda_jr_j(t)\Box^mT_\mu^\nu K_{\nu+m} (z,z_j)\Delta^{\nu+m+n/r}(\im z_j)\right|^qdV_{\beta+mq}(z)\\ &\lesssim& \|T_\mu^\nu\|^q\|f_t\|_{p,\alpha}^q\\ &\lesssim& \|T_\mu^\nu\|^q\|\{\lambda_j\}\|_{l_\alpha^p}^q.
\Eeas
That is
\Be\label{eq:Besovnormft1}
\int_{\mathcal D}\left|\sum_j\lambda_jr_j(t)\Box^mT_\mu^\nu K_{\nu+m} (z,z_j)\Delta^{\nu+m+n/r}(\im z_j)\right|^qdV_{\beta+mq}(z) \lesssim \|T_\mu^\nu\|^q\|\{\lambda_j\}\|_{l_\alpha^p}^q.
\Ee
By Kinchine's inequality, we have
\Beas
&&\int_0^1\left|\sum_j\lambda_jr_j(t)\Box^mT_\mu^\nu K_{\nu+m} (z,z_j)\Delta^{\nu+m+n/r}(\im z_j)\right|^qdt\\ &\ge& L_q\left(\sum_j|\lambda_j|^2|\Box^mT_\mu^\nu K_{\nu+m} (z,z_j)|^2\Delta^{2(\nu+m+n/r)}(\im z_j)\right)^{q/2}.
\Eeas
Integrating both sides of (\ref{eq:Besovnormft1}) from $0$ to $1$ with respect to $dt$, we obtain using Fubini's Theorem and the last inequality that
\Be\label{eq:Besovhalfestim}
\int_{\mathcal D}\left(\sum_j|\lambda_j|^2|\Box^mT_\mu^\nu K_{\nu+m} (z,z_j)|^2\Delta^{2(\nu+m+n/r)}(\im z_j)\right)^{q/2}dV_{\beta+mq}(z)\lesssim \|T_\mu^\nu\|^q\|\{\lambda_j\}\|_{l_\alpha^p}^q.
\Ee
Observe that for a sequence $\{a_j\}$ of complex numbers and for $\frac 2q< 1$,
$$\sum_j|a_j|^q\chi_{B_j}(z)\le \left(\sum_j|a_j|^2\chi_{B_j}(z)\right)^{\frac q2}.$$
Thus using the latter observation, we obtain from (\ref{eq:Besovhalfestim}),
\Beas
&& \sum_j|\lambda_j|^q\Delta^{q(\nu+m+n/r)}(\im z_j)\int_{B_j}|\Box^mT_\mu^\nu K_{\nu+m} (z,z_j)|^qdV_{\beta+mq}\\ &=& \int_{\mathcal D}\sum_j|\lambda_j|^q\Delta^{q(\nu+m+n/r)}(\im z_j)|\Box^mT_\mu^\nu K_{\nu+m} (z,z_j)|^q\chi_{B_j}(z)dV_{\beta+mq}\\ &\lesssim& \int_{\mathcal D}\left(\sum_j|\lambda_j|^2\Delta^{2(\nu+m+n/r)}(\im z_j)|\Box^mT_\mu^\nu K_{\nu+m} (z,z_j)|^2\right)^{q/2}dV_{\beta+mq}\\ &\lesssim& \|T_\mu^\nu\|^q\|\{\lambda_j\}\|_{l_\alpha^p}^q.
\Eeas
But from Lemma \ref{lem:meanvalue} we have
$$|\Box^mT_\mu^\nu K_{\nu+m} (z_j,z_j)|^q\lesssim \frac{1}{\Delta^{mq+\beta+n/r}(\im z_j)}\int_{B_j}|\Box^mT_\mu^\nu K_{\nu+m} (z,z_j)|^qdV_{\beta+mq}(z).$$
Thus
\Be\label{eq:sumboxdiscrit}\sum_j|\lambda_j|^q\Delta^{q(\nu+m+n/r)+mq+\beta+n/r}(\im z_j)|\Box^mT_\mu^\nu K_{\nu+m} (z_j,z_j)|^q\lesssim \|T_\mu^\nu\|^q\|\{\lambda_j\}\|_{l_\alpha^p}^q.
\Ee
Now note that
$$\Box^mT_\mu^\nu K_{\nu+m} (z,z_j)=C_m\int_{\mathcal D}K_{\nu+m}(w,z_j)K_{\nu+m}(z,w)d\mu(w).$$
Hence
\Be\label{eq:mudiscritestim}
\frac{\mu(B_j)}{\Delta^{2(\nu+m+n/r)}(\im z_j)}\lesssim \Box^mT_\mu^\nu K_{\nu+m} (z_j,z_j).
\Ee
Combining (\ref{eq:sumboxdiscrit}) and (\ref{eq:mudiscritestim}) we obtain
$$
\sum_j|\lambda_j|^q\left(\frac{\mu(B_j)}{\Delta^{\lambda(\gamma+n/r)}(\im z_j)}\right)^q\Delta^{\frac{q}{p}(\alpha+\frac{n}{r})}(\im z_j)\lesssim \|T_\mu^\nu\|^q\|\{\lambda_j\}\|_{l_\alpha^p}^q.
$$
Thus as the sequence $\{|\lambda_j|^q\Delta^{\frac{q}{p}(\alpha+\frac{n}{r})}(\im z_j)\}$ belongs to $l^{p/q}$, it follows by duality that the sequence $\{\left(\frac{\mu(B_j)}{\Delta^{\lambda(\gamma+n/r)}(\im z_j)}\right)^q\}$ belongs to $l^{p/p-q}$ which is the dual of $l^{p/q}$
by the sum pairing $$\langle \{a_j\},\{b_j\}\rangle_\nu:=\sum_{j}a_j\overline {b_j}.$$
Hence
$\sum\limits_j\left(\frac{\mu(B_j)}{\Delta^{\lambda(\gamma+n/r)}(\im z_j)}\right)^{\frac{pq}{p-q}}\lesssim\|T_\mu^\nu\|^q$
i.e.
$ \sum\limits_j\left(\frac{\mu(B_j)}{\Delta^{\lambda(\gamma+n/r)}(\im z_j)}\right)^{\frac{1}{1-\lambda}} \lesssim \|T_\mu^\nu\|^q$
 i.e.
 \Beas \sum_j\left(\frac{\mu(B_j)}{\Delta^{\gamma+n/r}(\im z_j)}\right)^{\frac{1}{1-\lambda}}\Delta^{\gamma+n/r}(\im z_j)&\lesssim& \|T_\mu^\nu\|^q.
\Eeas
Thus, $\{\hat{\mu}_\delta(z_j)\}\in\ell^{\frac{1}{1-\lambda}}_\gamma$ and from Lemma \ref{lem:integraldiscretizationAverBer} this means that $(\textrm{b})$ holds. The proof is complete.

\end{proof}
We obtain in the same way the following.
\btheo\label{theo:Toepbound4}
Let $1<q<q_\alpha$, $q<p<\infty$, $\alpha,\beta,\nu>\frac{n}{r}-1$ with $\nu>\frac{n}{r}-1+\frac{\beta-\frac{n}{r}+1}{q}-\frac{\alpha-\frac{n}{r}+1}{p}$. Define the numbers $\lambda=1+\frac{1}{p}-\frac{1}{q}$ and $\lambda \gamma=\nu+\frac{\alpha}{p}-\frac{\beta}{q}$. Assume $\nu$ is large enough so that $P_\nu$ is bounded on $L_\alpha^p(\mathcal D)$ and $1<q'<q_{\beta'}$ where $\beta'=\beta+(\nu-\beta)q'>\frac{n}{r}-1$. Then the following assertions are equivalent.
\begin{itemize}
\item[(a)] The operator $T_\mu^\nu$ extends as a bounded operator from $A_\alpha^p(\mathcal D)$ to $\mathbb {B}_\beta^q(\mathcal D)$.
\item[(b)] For any $\delta\in (0,1)$, the function $$\mathcal {D}\ni z\mapsto \frac{\mu(B_\delta(z))}{\Delta^{\gamma+\frac{n}{r}}(\im z)}$$
belongs to $L_\gamma^{1/(1-\lambda)}(\mathcal D)$.
\end{itemize}
\etheo

\section{Hankel operators from $A_\alpha^p(\mathcal D)$ to $\mathbb {B}_\beta^q(\mathcal D)$}
In this section, we provide criteria of boundedness of (small) Hankel operators from weighted Bergman spaces to weighted Besov spaces .
\subsection{Boundedness of $h_b: A^p_\nu(\mathcal D)\rightarrow \mathbb B^p_\nu(\mathcal D)$, $p>2$.}
We would like to start this section with a general result involving a change of weights. To avoid any ambiguity, we recall that we write $h_b^{(\nu)}f=P_\nu(b\overline f)$. The superscript may be removed when the result involves the same weight.
Let us also recall the notation $$q_\nu=\frac{\nu+\frac{n}{r}-1}{\frac{n}{r}-1}.$$
We have the following result.

\btheo Let $\nu, \mu>\frac{n}{r}-1$ and $p>\max
\left\{\frac{\nu+\frac{n}{r}-1}{\mu},\frac{\nu-\frac{n}{r}+1}{\mu-\frac{n}{r}+1}\right\}$.
Then the Hankel operator $h_b^{(\mu)}$ is bounded from
$A^p_\nu(\mathcal D)$ into $\mathbb B^p_\nu(\mathcal D)$ if
$b=P_\nu g$ for some $g\in L^\infty(\mathcal D)$ for which $P_\nu g$
makes sense.\etheo

\Proof We first observe that if $b=P_\nu g$ with $g\in
L^\infty(\mathcal D)$, then for $m>\frac{n}{r}-1$, as
$$\Delta^m\Box^m b=T_{\nu,m}g,$$
we have from Lemma \ref{lem:bergtypeinfty} that $$\sup_{z\in \mathcal D}\Delta^m(\im z)|\Box^mb(z)|<\infty.$$
Now, we prove the sufficient condition. We suppose that $b=P_\nu g$ with $g\in L^\infty(\mathcal D);$ recall that
$h_b^{(\mu)}(f)$ is a representative of a class in $\mathbb B^p_\nu(\mathcal D)$ is equivalent to saying that
 $\Delta^m\Box^mh_b^{(\mu)}(f)\in L_\nu^p(\mathcal D)$ with $m$ large enough. Let us put $k=\overline {f}\Delta^m\Box^mb$.
 Observe that for $f\in A^p_\nu(\mathcal D)$, $k\in L^p_\nu(\mathcal D)$ as $\Delta^m\Box^mb$ is bounded.
 Next, using the formula (\ref{parts}), one easily sees that
$$\Delta^m\Box^mh_b^{(\mu)}(f)=T_{\mu,m}k.$$
Since $m$ is taken large enough and
$p>\max \{\frac{\nu+\frac{n}{r}-1}{\mu},\frac{\nu-\frac{n}{r}+1}{\mu-\frac{n}{r}+1}\}$, Lemma \ref{lem:bergtype} gives that
$$||\Delta^m\Box^mh_b^{(\mu)}(f)||_{L_\nu^p}=||T_{\mu,m}k||_{L_\nu^p}\lesssim ||k||_{L_\nu^p}\lesssim ||g||_\infty||f||_{A_\nu^p}.$$
The proof is complete.
\ProofEnd
We now prove Theorem \ref{hankel 1}.


\begin{proof}[Proof of Theorem \ref{hankel 1}] The sufficiency is a special case of the previous theorem
corresponding to $\mu=\nu$. Conversely, if $h_b$ is bounded, then
for any $f\in A^p_\nu(\mathcal D)$ and any $g\in A^{p'}_\nu(\mathcal
D)$, using the reproduction formula (\ref{rep3}), we have
\begin{equation}\label{eq:dualityHank} |\langle
h_bf,g\rangle_{\nu,m}| = |\langle b,fg\rangle_{\nu,m}|\le
C||f||_{A_\nu^p}||g||_{A_\nu^{p'}}.\end{equation}
We take $f(z)=f_w(z)=\Delta^{-(\nu+\frac{m}{2})}(\frac{z-\overline
w}{i})$ and
$g(z)=g_w(z)=\Delta^{-(\frac{n}{r}+\frac{m}{2})}(\frac{z-\overline
w}{i})$ with $m$ large enough. Following Lemma \ref{int-compl}, we
have
$$||f||_{A_\nu^p}=C_{\nu,m,p}\Delta^{-(\nu+\frac{m}{2})+(\nu+\frac{n}{r})\frac{1}{p}}(\im w),$$
and
$$||g||_{A_\nu^{p'}}=C_{\nu,m,p}\Delta^{-(\frac{n}{r}+\frac{m}{2})+(\nu+\frac{n}{r})\frac{1}{p'}}(\im
w).$$ Now, since by the reproduction formula (\ref{rep3}), we have
$P_{\nu+m}g=g,$ replacing $f$ and $g$ in (\ref{eq:dualityHank}), we
obtain
$$\Delta^m(\im w)\left|\int_{\mathcal
D}K_{\nu+m}(w,z)\Box^mb(z)dV_{\nu+m}(z)\right|\le C.$$ Taking
$\ell=m$ in (\ref{repr}), this is equivalent to
$$\Delta^m(\im w)|\Box^mb(w)|=\Delta^m(\im w)|P_{\nu+m}(\Box^mb)(w)|\le C.$$
Let $h=\Delta^m\Box^mb$, we have just obtained that $h\in L^\infty (\mathcal D)$. Moreover,
as $b\in A_\nu^2(\mathcal D)$, $P_\nu h$ makes sense and we have from the reproducing formula, taking $\ell=0$ in (\ref{repr})
that
$$P_\nu h=\int_{\mathcal D}K_\nu(\cdot,w)\Delta^m(\im w)\Box^mb(w)dV_\nu(w)=b.$$
The proof is complete.
\end{proof}

As a consequence of the above result, we have the
following corollary.

\bcor\label{cor:ApAp} Let $1<p<\infty$ and
$\nu>\frac{n}{r}-1$. If $P_\nu$ is bounded on
$L_\nu^p(\mathcal D)$ then the Hankel operator $h_b$
extends as a bounded operator on $A_\nu^p(\mathcal D)$ if
and only if $b=P_\nu g$ for some $g\in L^\infty(\mathcal D)$ for which $P_\nu g$ makes sense.
\ecor

 When $p=\infty$, we
can prove in the same way the following result.

\btheo The Hankel operator $h_b$ is bounded from $\mathcal
H^\infty(\mathcal D)$ into $\mathbb B$ if and only if $b=P_\nu g$ for some $g\in L^\infty(\mathcal D)$ for which $P_\nu g$ makes sense.\etheo

\subsection{Boundedness of $h_b:A^p_\alpha(\mathcal D)\rightarrow \mathbb B^q_\beta(\mathcal D)$, $1\le p<q<\infty$.}
Let us first recall the following embedding result from \cite{DD}.
\blem\label{lem:esti} Let $1\le p<q<\infty$ and
$\alpha, \beta>\frac{n}{r}-1$. Assume that $\frac{1}{p}(\alpha+\frac{n}{r})=\frac{1}{q}(\beta+\frac{n}{r})$. Then there exists a positive constant
$C$ such that for any $f\in A_\alpha^p(\mathcal D)$,
$$\int_{\mathcal
D}|f(z)|^q\Delta^{\beta-\frac{n}{r}}(\im
z)dV(z)\le C||f||_{p,\alpha}^q.$$
\elem

\Proof
Let us recall that there is a constant $C>0$ such that for any $f\in A_\alpha^p(\mathcal D)$ the following
pointwise estimate holds:
\begin{equation}|f(z)|\le C\Delta^{-\frac{1}{p}(\alpha+\frac{n}{r})}(\im z)\|f\|_{p,\alpha},\,\,\,\textrm{for all}\,\,\, z\in \mathcal {D}.
\end{equation}
It follows easily that
\Beas
I &:=& \int_{\mathcal
D}|f(z)|^q\Delta^{\beta-\frac{n}{r}}(\im
z)dV(z)\\ &=& \int_{\mathcal
D}|f(z)|^p|f(z)|^{q-p}\Delta^{\beta-\frac{n}{r}}(\im
z)dV(z)\\ &\le& C\|f\|_{A_\alpha^p}^{q-p}\int_{\mathcal
D}|f(z)|^pdV_\alpha(z)= C\|f\|_{p,\alpha}^q.
\Eeas
\ProofEnd
We can now prove Theorem \ref{hankel 2}.

\begin{proof}[Proof of Theorem \ref{hankel 2}] We first prove the sufficiency. Let $f\in
A_\alpha^p(\mathcal D)$. From the integration by parts
formula, we obtain \Beas\Box^mh_b^{(\nu)}f(z) &=&
C_{\nu,m}\int_{\mathcal
D}K_{\nu+m}(z,w)b(w)\overline {f(w)}dV_\nu(w)\\
&=& C_{\nu,m}\int_{\mathcal
D}K_{\nu+m}(z,w)\Box^mb(w)\overline {f(w)}dV_{\nu+m}(w)\\
&=& P_{\nu+m}T_bf(z),\Eeas where the operator $T_b$ is
defined by $$T_bf(z)=C_{\nu,m}\overline
{f(z)}\Box^mb(z),\,\,\,z\in \mathcal D.$$ Let us prove that
$T_b$ is bounded from $A_\alpha^p(\mathcal D)$ to
$L_{\beta+mq}^q(\mathcal D)$. Observe that if $\vartheta=-\frac{q}{s'}(\mu+\frac{n}{r})+\beta$, then $\frac{1}{q}(\vartheta+\frac{n}{r})=\frac{1}{p}(\alpha+\frac{n}{r})$. Using the embedding of
$A_\alpha^p(\mathcal D)$ into
$L_{\vartheta}^q(\mathcal
D)$, we easily obtain \Beas\int_{\mathcal
D}|T_bf(z)|^qdV_{\beta+mq}(z) &=& C\int_{\mathcal
D}|f(z)|^q|\Box^mb(z)|^qdV_{\beta+mq}(z)\\ &\le&
C\int_{\mathcal
D}|f(z)|^q\Delta^{-\frac{q}{s'}(\mu+\frac{n}{r})}(\im
z)dV_{\beta}(z)\\ &\le& C||f||_{p,\alpha}^q.\Eeas Now, since
by Lemma \ref{lem:bergtype}, for $m$ large enough and for $q$ as in our assumptions, $P_{\nu+m}$ is bounded on
$L_{\beta+mq}^q(\mathcal D)$, we conclude that for any $f\in
A_\alpha^p(\mathcal D)$, $\Box^mh_bf\in A_{\beta+mq}^q(\mathcal
D)$.

\vskip .2cm
For the converse, we consider the two different situations:
$1\le q\le 2$ and $2<q<\infty$.

\medskip

{\bf Case $2<q<\infty$}. If $h_b^{(\nu)}$ is bounded then as we
have seen in the previous section, there is a positive
constant $C$ such that for any $f\in A_\alpha^p(\mathcal D)$
and any $g\in A_{\beta'}^{q'}(\mathcal D)$ inequality
(\ref{eq:dualityHank}) holds. We take again
$f(z)=f_w(z)=\Delta^{-(\nu+\frac{m}{2})}(\frac{z-\overline
w}{i})$ and
$g(z)=g_w(z)=\Delta^{-(\frac{n}{r}+\frac{m}{2})}(\frac{z-\overline
w}{i})$ with $m$ large enough. From Lemma \ref{int-compl},
we have
$$||f||_{p,\alpha}=C_{\nu,m,p}\Delta^{-(\nu+\frac{m}{2})+(\alpha+\frac{n}{r})\frac{1}{p}}(\im w),$$\ProofEnd
and
$$||g||_{A_{\beta'}^{q'}}=C_{\nu,m,q}\Delta^{-(\frac{n}{r}+\frac{m}{2})+(\beta'+\frac{n}{r})\frac{1}{q'}}(\im
w).$$ Replacing $f$ and $g$ in (\ref{eq:dualityHank}), we
obtain
$$\Delta^{m-\frac{1}{p}(\alpha+\frac{n}{r})-\frac{1}{q'}(\beta'+\frac{n}{r})+\nu+\frac{n}{r}}(\im
w)\left|\int_{\mathcal
D}K_{\nu+m}(w,z)\Box^mb(z)dV_{\nu+m}(z)\right|\le C.$$ By
(\ref{repr}), this is equivalent to
$$
\Delta^{m+\frac{1}{s'}(\mu+\frac{n}{r})}(\im
w)|\Box^mb(w)| \le C,$$ for any $w\in \mathcal
D$.

\medskip
{\bf Case $1\le q\le 2$}. Note that in this case $\mathbb
B_\beta^q(\mathcal D)=A_\beta^q(\mathcal D)$. That $h_b^{(\nu)}$ is bounded, is
equivalent to saying that there exists a constant $C>0$ so that for
any $f\in A_\alpha^p(\mathcal D)$ and any representative $g$ of any
class in $g\in A_{\beta'}^{q'}(\mathcal
D)$,\begin{equation}\label{eq:dualitypetitq}|\langle h_b^{(\nu)}f,
g\rangle_{\nu}|\le
C||f||_{p,\alpha}||g||_{q',\beta'}\end{equation} 
Now, taking $g(z)=g_\zeta(z)=K_\nu(z,\zeta)$ and
$f(z)=f_\zeta(z)=\Delta^{-m}(\frac{z-\overline {\zeta}}{i})$,
we obtain
\Beas \langle h_b^{(\nu)}f, g\rangle_{\nu} &=& \int_{\mathcal
D}\left(h_b^{(\nu)}f(z)\right)\overline {g(z)}dV_{\nu}(z)\\ &=&
C_{\nu}\int_{\mathcal
D}K_{\nu}(\zeta,z)\left(h_b^{(\nu)}f(z)\right)dV_{\nu}(z)\\ &=& C_{\nu}P_{\nu}(h_b^{(\nu)}f)(\zeta)\\
&=& C_{\nu}h_b^{(\nu)}f(\zeta)=C_{\nu}\int_{\mathcal D}K_\nu(\zeta,z)\Box^mb(z)\overline {f_\zeta}(z)dV_{\nu+m}(z)\\
&=& C_{\nu}\int_{\mathcal D}K_{\nu+m}(\zeta,z)\Box^mb(z)dV_{\nu+m}(z)\\ &=& C_{\nu,m}\Box^mb(\zeta).
\Eeas

Now, choosing $m$ large enough, we
obtain from Lemma \ref{int-compl} that
$$||f||_{p,\alpha}=C_{\nu,m,p}\Delta^{-m+(\alpha+\frac{n}{r})\frac{1}{p}}(\im \zeta)$$
and since the conditions on $q$ insure that $g\in A^{q'}_{\beta'}$,
$$||g||_{q',\beta.}=C_{\nu,m,q}\Delta^{-(\nu+\frac{n}{r})+\frac{1}{q'}(\beta'+\frac{n}{r})}(\im
\zeta).$$ Taking all the above observations in
(\ref{eq:dualitypetitq}), we obtain that
$$\Delta^{m+(\nu+\frac{n}{r})-\frac{1}{p}(\alpha+\frac{n}{r})-\frac{1}{q'}(\beta'+\frac{n}{r})}(\im \zeta)|\Box^mb(\zeta)|\le C<\infty,\,\,\,\zeta\in \mathcal D.$$
The proof is complete.
\end{proof}

\bcor Let $1\le p<q<\infty$ and $\nu>\frac{n}{r}-1$. If $P_\nu$ is
bounded on $L_\nu^q(\mathcal D)$, then the Hankel operator $h_b$
extends into a bounded operator from $A_\nu^p(\mathcal D)$ to
$A^q_\nu(\mathcal D)$ if and only for some $m$ large,
$$\Delta^{m-(\nu+\frac{n}{r})(\frac{1}{p}-\frac{1}{q})}\Box^mb\in
L^\infty.$$\ecor

\subsection{Boundedness of $h_b^{(\nu)}:A^p_\alpha(\mathcal D)\rightarrow \mathbb B^q_\beta(\mathcal D)$, $1<q<p<\infty$.}
We begin with the case $p=\infty$.

\btheo Let $2\leq p<\infty$, $\nu>\frac{n}{r}-1$, and $\alpha\in \mathbb{R}$. Assume that 
$$\max\{\frac{\alpha+\frac{n}{r}-1}{\nu},\frac{\alpha-\frac{n}{r}+1}{\nu-\frac{n}{r}+1}\}<p.$$ Then the Hankel operator $h_b^{(\nu)}$ extends into a bounded operator from $\mathcal
H^\infty(\mathcal D)$ to $\mathbb B^p_\alpha(\mathcal D)$ if and only
if $b=P_\alpha^{(m)}(f)$ for some $f\in L^p_\alpha(\mathcal D)$ and $m$ a large enough integer.
\etheo

\begin{proof} First suppose that for $m$ large, $\Delta^m\Box^mb\in
L_\alpha^p(\mathcal D)$. Then for any $f\in \mathcal H^\infty(\mathcal
D)$ and any $g\in A^{p'}_{\alpha'}(\mathcal D)$, $\alpha'=\alpha+(\nu-\alpha)p'>\frac{n}{r}-1$, we easily have \Beas
|\langle h_b^{(\nu)}f,g\rangle_{\nu,m}| &=& \left|\int_{\mathcal
D}\overline{f(z)g(z)}\Box^mb(z)dV_{\nu+m}(z)\right|\\ &\le&
||\Delta^m\Box^mb||_{p,\alpha}||f||_{\infty}||g||_{p',\alpha'}<\infty.\Eeas

Conversely, if $h_b^{(\nu)}$ is bounded then as before there exists $C>0$
such that for any $f\in \mathcal H^\infty(\mathcal D)$ and any $g\in
A_{\alpha'}^{p'}(\mathcal D)$, $$ |\langle h_b^{(\nu)}f,g\rangle_{\nu,m}| =
\left|\int_{\mathcal
D}\overline{f(z)g(z)}\Box^mb(z)dV_{\nu+m}(z)\right|\le
C||f||_{\infty}||g||_{p',\alpha'}.$$ Taking $f(z)=1$ for any $z\in
\mathcal D$, we obtain that there exists a constant $C>0$ such that
for any $g\in A_{\alpha'}^{p'}(\mathcal D)$,
$$\left|\int_{\mathcal
D}\overline{g(z)}\Box^mb(z)dV_{\nu+m}(z)\right|\le
C||g||_{p',\alpha'};$$ i.e. $$|\langle b,g\rangle_{\nu,m}|\le
C||g||_{p',\alpha'}.$$ It follows from Lemma \ref{lem:ApBp0} that $b$ is
a representative of a class in $\mathbb B^p_\alpha(\mathcal
D)=(A_{\alpha'}^{p'}(\mathcal D))^*$ with respect to the pairing $\langle \cdot,\cdot\rangle_{\nu,m}$. As the condition $p>\max\left\{\frac{\alpha+\frac{n}{r}-1}{\nu},\frac{\alpha-\frac{n}{r}+1}{\nu-\frac{n}{r}+1}\right\}$ implies that $P_{\nu+m}$ is bounded on $L_{\alpha+mp}^p(\mathcal{D})$, for $m$ large enough, the equality (\ref{eq:extendedBergproj1}) allows us to conclude that
$b=cP_\alpha^{(m)}(f)$ with $f=\Delta^m\Box^mb$.
The proof is complete.
\end{proof}

We have in particular the following.

\bcor Let $1<p<\infty$ and $\nu>\frac{n}{r}-1$. If $P_\nu$
is bounded on $L_\nu^p(\mathcal D)$, then the Hankel
operator $h_b$ extends into a bounded operator from
$\mathcal H^\infty(\mathcal D)$ to $A^p_\nu(\mathcal D)$ if
and only if $b$ is in
$A^p_\nu(\mathcal D)$.\ecor

In the same way, with the help of Theorem \ref{theo:dualitychange}, we obtain the following.
\btheo Let $\max\left\{q_\alpha',\frac{\alpha+\frac{n}{r}-1}{\nu},\frac{\alpha-\frac{n}{r}+1}{\nu-\frac{n}{r}+1}\right\}<p<q_\alpha$, $\alpha, \nu>\frac{n}{r}-1$.  Then the Hankel operator $h_b^{(\nu)}$ extends into a bounded operator from $\mathcal
H^\infty(\mathcal D)$ to $A^p_\alpha(\mathcal D)$ if and only
if $b\in A^p_\alpha(\mathcal D)$.
\etheo
The following endpoint case follows also the same using the duality result in Lemma \ref{A1}.
\btheo Let $\nu>\frac{n}{r}-1$. Then the Hankel operator $h_b^{(\nu)}$ extends into a bounded operator from $\mathcal
H^\infty(\mathcal D)$ to $A^1_\nu(\mathcal D)$ if and only
if $b\in A^1_\nu(\mathcal D)$.
\etheo

We next prove the following lemma.
\blem\label{lem:boxdiscrit}
Let $2\le q<\infty$, $\alpha,\nu>\frac{n}{r}-1$, $\mu\in \mathbb{R}$. Assume that $\frac{\mu-\frac{n}{r}+1}{\nu-\frac{n}{r}+1}<q<\infty$. For $m$ an integer such that $P_\gamma$ is bounded on $L_\beta^{q'}$, $\gamma=\alpha+\nu+m+\frac{n}{r},\,\,\, \beta=\mu+(\nu-\mu)q',\,\,\,\frac{1}{q}+\frac{1}{q'}=1$, consider the operator
$$T_{\alpha,\nu}^mf(z)=\int_{\mathcal D}K_{\alpha+\nu+m+\frac{n}{r}}(z,w)\Box^m f(w)dV_{\nu+m}(w)$$
define for functions with compact support. Then if there exists a constant $C>0$ such that for any $\delta\in (0,1)$ and any $\delta$-lattice $\{z_j\}_{j\in \mathbb N_0}$
of points in $\mathcal D$,
$$\sum_j\Delta^{q(\alpha+\frac{n}{r})+\mu+\frac{n}{r}}(\im z_j)|\Delta^m(\im z_j)T_{\alpha,\nu}^mf(z_j)|^q\le C^q,$$
then $f$ is a representative of a class $F\in \mathbb {B}_{\mu}^q$ with $\|F\|_{\mathbb {B}_\mu^q}\le C$.
\elem
\Proof
First note that the condition on the exponent $q$ implies that $\beta>\frac{n}{r}-1$, and that $A_\beta^{q'}$ is the dual space of $\mathbb {B}_\mu^q$ under the pairing
$$\langle h,F\rangle_{\nu,m}=\int_{\mathcal D}h(z)\overline {\Box^mF(z)}dV_{\nu+m},\,\,\,h\in A_\beta^{q'}\,\,\,\textrm{and}\,\,\,F\in \mathbb {B}_\mu^q.$$
As the projector $P_\gamma$ ($\gamma=\alpha+\nu+m+\frac{n}{r}$) is bounded on $L_\beta^{q'}(\mathcal D)$, by Theorem \ref{theo:atomdecompo}, any $h\in A_\beta^{q'}(\mathcal D)$ can be represented as
$$h(z)=\sum_j\lambda_j\Delta^{\gamma+\frac{n}{r}}(\im z_j)K_\gamma(z,z_j)$$ with $\|\{\lambda_j\}\|_{l_\beta^{q'}}\lesssim \|h\|_{q',\beta}$. It follows that
if $F$ is a class represented by $f$,
\Beas
\langle h,F\rangle_{\nu.m} &=& \int_{\mathcal D}\sum_j\lambda_j\Delta^{\gamma+\frac{n}{r}}(\im z_j)K_\gamma(z,z_j)\overline {\Box^mF(z)}dV_{\nu+m}(z)\\ &=&
\sum_j\lambda_j\Delta^{\alpha+\nu+m+2\frac{n}{r}}(\im z_j)\overline {T_{\alpha,\nu}^mf(z_j)}.
\Eeas
It follows from H\"older's inequality that
\Beas
\|F\|_{\mathbb {B}_\mu^q} &:=& \sup_{h\in A_\beta^{q'}, \|h\|_{q',\beta}\le 1}|\langle h,F\rangle_{\nu,m}|\\ &\le& \sup_{h\in A_\beta^{q'}, \|h\|_{q',\beta}\le 1}\|\{\lambda_j\}\|_{l_\beta^{q'}}\left(\sum_j\Delta^{q(\alpha+\frac{n}{r})}(\im z_j)|\Delta^m(\im z_j)T_{\alpha,\nu}^mf(z_j)|^q\Delta^{\mu+\frac{n}{r}}(\im z_j)\right)^{1/q}\\ &\lesssim& C.
\Eeas
\ProofEnd

The following can be proved the same way with the help of the duality in Theorem \ref{theo:dualitychange}.
\blem\label{lem:boxdiscrit1}
Let $\max\left\{q_\beta',\frac{\mu+\frac{n}{r}-1}{\nu},\frac{\mu-\frac{n}{r}+1}{\nu-\frac{n}{r}+1}\right\}< q<q_\mu$, $\alpha,\mu,\nu>\frac{n}{r}-1$. Define $\beta=\mu+(\nu-\mu)q',\,\,\,\frac{1}{q}+\frac{1}{q'}=1$, and consider the operator
$$T_{\alpha,\nu}f(z)=\int_{\mathcal D}K_{\alpha+\nu+\frac{n}{r}}(z,w)f(w)dV_{\nu}(w)$$
define for functions with compact support. Then if there exists a constant $C>0$ such that for any $\delta\in (0,1)$ and any $\delta$-lattice $\{z_j\}_{j\in \mathbb N_0}$
of points in $\mathcal D$,
$$\sum_j\Delta^{q(\alpha+\frac{n}{r})+\mu+\frac{n}{r}}(\im z_j)|T_{\alpha,\nu}f(z_j)|^q\le C^q,$$
then $f\in A_{\mu}^q$ with $\|f\|_{A_\mu^q}\le C$.
\elem
Note that the condition  $\max\left\{1,\frac{\mu+\frac{n}{r}-1}{\nu},\frac{\mu-\frac{n}{r}+1}{\nu-\frac{n}{r}+1}\right\}\le q<q_\mu$ provides that $\beta>\frac{n}{r}-1$ and that $P_\nu$ is bounded on $A_\mu^q(\mathcal{D})$. Also as $1<q'<q_\beta$, by Theorem \ref{theo:dualitychange}, $A_\beta^{q'}(\mathcal{D})$ is the dual of $A_\mu^q(\mathcal{D})$ under the pairing $\langle \cdot,\cdot\rangle_\nu$.
\vskip .2cm
Let us know prove Theorem \ref{hankel 3}.

\begin{proof}[Proof of Theorem \ref{hankel 3}] We start with the sufficiency which is the harmless part. Let $m$ be a positive integer such that Hardy inequality holds
for $(q,\beta+mq)$. Recall that $d\lambda(z)=\Delta^{-2n/r}(\im z)dV(z)$. Then for
any $f\in A_{\alpha}^p(\mathcal D)$ and any $g\in A_{\beta'}^{q'}(\mathcal
D)$,
\Beas |\langle h_b^{(\nu)}f,g\rangle_{\nu,m}| &=& \left|\int_{\mathcal
D}\overline{f(z)g(z)}\Box^mb(z)dV_{\nu+m}(z)\right|\\ &=& \left|\int_{\mathcal
D}[\Delta^{\frac{1}{s}(\gamma+\frac{n}{r})}(\im z)\overline{f(z)g(z)}][\Delta^{m+\frac{1}{s'}(\mu+\frac{n}{r})}(\im z)\Box^mb(z)]d\lambda(z)\right|\\ &\leq& \|fg\|_{s,\gamma}\|\Delta^m\Box^m b\|_{s',\mu}\\ &\le&
||\Delta^m\Box^mb||_{s',\mu}||f||_{p,\alpha}||g||_{q',\beta'}<\infty.
\Eeas
Thus $$\|h_b^{(\nu)}\|:=\sup_{||f||_{A_\alpha^p}\le 1, ||g||_{A_{\beta'}^{q'}}\le 1}|\langle h_b^{(\nu)}f,g\rangle_{\nu,m}|\le ||\Delta^m\Box^mb||_{s',\mu}.$$
Let us now move to the necessity. We observe that as  $P_{\sigma}$ is bounded on $L_\alpha^p(\mathcal D)$,
we have from Theorem \ref{theo:atomdecompo} that given a $\delta$-lattice $\{z_j\}_{j\in \mathbb N_0}$, $\delta\in (0,1)$,
for every complex sequence $\{\lambda_j\}\in l_{\alpha}^p$, the series $$\sum_j\lambda_jK_{\sigma} (z,z_j)\Delta^{\sigma+n/r}(\im z_j)$$ is convergent in $A_\alpha^p(\mathcal D)$ and it sum $f$ satisfies
$$\|f\|_{p,\alpha}\le C\left(\sum_j|\lambda_j|^p\Delta^{\alpha+n/r}(\im z_j)\right)^{1/p}.$$
Thus if $r_j(t)$ is a sequence of Rademacher functions and $\{\lambda_j\}\in l_{\alpha}^p$,
$$f_t(z)=\sum_j\lambda_jr_j(t)K_{\sigma} (z,z_j)\Delta^{\sigma+n/r}(\im z_j)$$ also converges in
$A_{\alpha}^p(\mathcal D)$ with $\|f_t\|_{p,\alpha}\le C\|\{\lambda_j\}\|_{l_\alpha^p}$.
\vskip .1cm
As $h_b^{(\nu)}$ extends as a bounded operator from $A_\alpha^p(\mathcal D)$ to $\mathbb {B}_\beta^q(\mathcal D)$, we obtain for $m$ integer large enough,
\Beas
\|\Box^mh_b f_t\|_{q,\beta+mq}^q &=& \int_{\mathcal D}\left|\sum_j\lambda_jr_j(t)\Box^mh_b K_{\sigma} (z,z_j)\Delta^{\sigma+n/r}(\im z_j)\right|^qdV_{\beta+mq}(z)\\ &\lesssim& \|h_b^{(\nu)}\|^q\|f_t\|_{p,\alpha}^q\\ &\lesssim& \|h_b^{(\nu)}\|^q\|\{\lambda_j\}\|_{l_\alpha^p}^q.
\Eeas
That is
\Be\label{eq:Besovnormft}
\int_{\mathcal D}|\sum_j\lambda_jr_j(t)\Box^mh_b K_{\sigma} (z,z_j)\Delta^{\sigma+n/r}(\im z_j)|^qdV_{\beta+mq}(z) \lesssim \|h_b^{(\nu)}\|^q\|\{\lambda_j\}\|_{l_\alpha^p}^q.
\Ee
Following the same reasoning as in the proof of the necessity part in Theorem \ref{theo:Toepbound2} up to the inequality (\ref{eq:sumboxdiscrit}), we obtain that
\Be\label{eq:sumboxdiscrithank}\sum_j|\lambda_j|^q\Delta^{q(\sigma+n/r)+mq+\beta+n/r}(\im z_j)|\Box^mh_b^{(\nu)} K_{\sigma} (z_j,z_j)|^q\lesssim \|h_b^{(\nu)}\|^q\|\{\lambda_j\}\|_{l_\alpha^p}^q.
\Ee
We next observe that
\Beas\Box^mh_b^{(\nu)} K_{\sigma} (z,z_j) &=& C_m\int_{\mathcal D}b(w)K_{\sigma}(z_j,w)K_{\nu+m}(z,w)dV_\nu(w)\\ &=& C_m\int_{\mathcal D}\Box^mb(w)K_{\sigma}(z_j,w)K_{\nu+m}(z,w)dV_{\nu+m}(w).
\Eeas
Thus
$$\Box^mh_b^{(\nu)} K_{\sigma} (z_j,z_j)=C_m\int_{\mathcal D}\Box^mb(w)K_{\sigma+\nu+m+\frac{n}{r}}(z_j,w)dV_{\nu+m}(w)=C_mT_{\sigma,\nu}^mb(z_j).$$
Taking this in (\ref{eq:sumboxdiscrithank}) we obtain
\Be\label{eq:sumboxdiscrithank}\sum_j|\lambda_j|^q\Delta^{q(\sigma+n/r)+mq+\beta+n/r}(\im z_j)|T_{\sigma,\nu}^mb(z_j)|^q\lesssim \|h_b^{(\nu)}\|^q\|\{\lambda_j\}\|_{l_\alpha^p}^q.
\Ee
As $\{\lambda_j\}$ is chosen arbitrary in $l_\alpha^{p}$, it follows by duality since $\{|\lambda_j|^q\}$ belongs to $l_\alpha^{p/q}$ that the sequence
$\{\Delta^{q(\sigma+n/r)+mq+\beta-\nu}(\im z_j)|T_{\sigma,\nu}^mb(z_j)|^q\}$ belongs to $l_{\alpha'}^{p/(p-q)}$, $\alpha'=\alpha+\frac{p}{p-q}(\nu-\alpha)$, which is the dual of $l_\alpha^{p/q}$ under the sum
pairing $$\langle \{a_j\},\{b_j\}\rangle_\nu:=\sum_ja_j\overline {b_j}\Delta^{\nu+n/r}(\im z_j)$$ with
$$\|\{\Delta^{q(\sigma+n/r)+mq+\beta-\nu}(\im z_j)|T_{\sigma,\nu}^mb(z_j)|^q\}\|_{l_{\alpha'}^{p/(p-q)}}\lesssim \|h_b^{(\nu)}\|^q.$$ That is
$$\sum_j\Delta^{\frac{pq}{p-q}(\alpha+n/r)+\mu+n/r}(\im z_j)|\Delta^m(\im z_j)T_{\sigma,\nu}^mb(z_j)|^{\frac{pq}{p-q}}\lesssim \|h_b^{(\nu)}\|^q.$$
Thus as $\frac{pq}{p-q}>q>2$, using Lemma \ref{lem:boxdiscrit} we conclude that $b$ is a representative of a class $B\in \mathbb {B}_\mu^{s'}$, $s'=\frac{pq}{p-q}$, and
$$\|B\|_{\mathbb {B}_\mu^{s'}}\lesssim \|h_b^{(\nu)}\|.$$
The proof is complete.
\end{proof}

We obtain in the same way using the duality result in Theorem \ref{theo:dualitychange} and Lemma \ref{lem:boxdiscrit}, the following.
\btheo Let $\max\left\{q_\beta',\frac{\beta+\frac{n}{r}-1}{\nu},\frac{\beta-\frac{n}{r}+1}{\nu-\frac{n}{r}+1}\right\}< q<q_\beta$ and $\alpha, \beta, \nu>\frac{n}{r}-1$, $\frac{1}{p}+\frac{1}{p'}=\frac{1}{q}+\frac{1}{q'}=1$.  Define
$\beta'=\beta+(\nu-\beta)q'$; $\frac{1}{2}<\frac{1}{p}+\frac{1}{q'}=\frac{1}{s}<1$; $\frac{\alpha}{p}+\frac{\beta'}{q'}=\frac{\gamma}{s}$; $\frac{1}{s}+\frac{1}{s'}=1$ and $\frac{\gamma}{s}+\frac{\mu}{s'}=\nu$.
Assume that \begin{equation}\label{mugrand}
\frac{1}{p}(\alpha-\frac{n}{r}+1)+\frac{1}{q'}(\beta'-\frac{n}{r}+1)<\nu-\frac{n}{r}+1.
\end{equation}
Then the followig assertions hold.
\begin{itemize}
\item[(i)] If $b$ is the representative of a class in $\mathbb{B}_\mu^{s'}$, then the Hankel operator $h_b^{(\nu)}$ extends into a bounded operator from $A_\alpha^p(\mathcal D)$ to $A^q_\beta(\mathcal D)$.
\item[(ii)] If there exists $\sigma>\frac{n}{r}-1$ such that $P_\sigma$ is bounded on $L_\alpha^p(\mathcal{D})$ and if $h_b^{(\nu)}$ extends into a bounded operator from $A_\alpha^p(\mathcal D)$ to $A^q_\beta(\mathcal D)$, then $b$ is the representative of a class in $\mathbb{B}_\mu^{s'}$.
\end{itemize}
\etheo
Using Theorem \ref{theo:dualitychange} and Lemma \ref{lem:boxdiscrit1}, we obtain in the same way the following.
\btheo Let $\max\left\{q_\beta',\frac{\beta+\frac{n}{r}-1}{\nu},\frac{\beta-\frac{n}{r}+1}{\nu-\frac{n}{r}+1}\right\}< q<q_\beta$ and $\alpha, \beta, \nu>\frac{n}{r}-1$, $\frac{1}{p}+\frac{1}{p'}=\frac{1}{q}+\frac{1}{q'}=1$.  Define
$\beta'=\beta+(\nu-\beta)q'$; $\frac{1}{p}+\frac{1}{q'}=\frac{1}{s}$; $\frac{\alpha}{p}+\frac{\beta'}{q'}=\frac{\gamma}{s}$; $\frac{1}{s}+\frac{1}{s'}=1$ and $\frac{\gamma}{s}+\frac{\mu}{s'}=\nu$.
Assume that \begin{equation}\label{mugrand}
\frac{1}{p}(\alpha-\frac{n}{r}+1)+\frac{1}{q'}(\beta'-\frac{n}{r}+1)<\nu-\frac{n}{r}+1,
\end{equation}
and 
\begin{equation}\label{sgrand}
\frac{1}{q_\gamma}<\frac{1}{s}<\max\left\{\frac{1}{q_\mu'},\frac{\nu}{\mu+\frac{n}{r}-1},\frac{\nu-\frac{n}{r}+1}{\mu-\frac{n}{r}+1}\right\}.
\end{equation}
Then the following hold.
\begin{itemize}
\item[(i)] If $b\in A_\mu^{s'}$, then the Hankel operator $h_b^{(\nu)}$ extends into a bounded operator from $A_\alpha^p(\mathcal D)$ to $A^q_\beta(\mathcal D)$.
\item[(ii)] If there exists $\sigma>\frac{n}{r}-1$ such that $P_\sigma$ is bounded on $L_\alpha^p(\mathcal{D})$ and if $h_b^{(\nu)}$ extends into a bounded operator from $A_\alpha^p(\mathcal D)$ to $A^q_\beta(\mathcal D)$, then $b\in A_\mu^{s'}$.
\end{itemize}
\etheo
In Particular, we have the following result.

\bprop Let $1< q < p<\infty$ and $\nu>\frac{n}{r}-1$. Suppose that
$P_\nu$ is bounded on both $L_\nu^p(\mathcal D)$ and $L_\nu^q(\mathcal D)$. Then the Hankel operator $h_b$ is
bounded from $A_\nu^p(\mathcal D)$ to $A^q_\nu(\mathcal D)$ if and only if $b\in \mathbb B^s_\nu(\mathcal D)$,
where $s=\frac{pq}{p-q}$.
\eprop
\begin{proof}
We note that $q<\frac{pq}{p-q}=s$. If $s\ge 2$, everything follows as in the proof of the previous theorem. If $q<s<2$, then by interpolation
we also have that $P_\nu$ is bounded on $L_\nu^s(\mathcal D)$ and again the proof follows as for the theorem above. Note that in this last case, $k_0(s,\nu)=0$. The proof is complete.
\end{proof}

\section{Weak Factorization of functions in Bergman spaces of tube domains over symmetric cones}
In this section, we prove Theorem \ref{thm:weakfact1} and Theorem \ref{thm:weakfact2}. In fact we only have to prove equivalence between boundedness of Hankel operators and and weak factorization results as stated in the two theorems. We have the following result for weighted Bergman spaces of our setting.
\btheo\label{prop:equivhankweakfact}
Let $\gamma,\nu>\frac nr-1$ and $1<s<q_\gamma.$ Let $1<p, q<\infty$ and $\alpha,\beta>\frac nr-1$
so that $\frac{1}{s}=\frac{1}{p}+\frac{1}{q}<1$, $\frac{\gamma}{s}=\frac{\alpha}{p}+\frac{\beta}{q}$ and $$\frac{1}{p}(\alpha-\frac{n}{r}+1)+\frac{1}{q}(\beta-\frac{n}{r}+1)<\nu-\frac{n}{r}+1$$ holds. Assume that $P_\sigma$ is bounded on $A_\alpha^p(\mathcal{D})$ for some $\sigma>\frac{n}{r}-1$, and that 
\begin{equation}\label{weakfactocond1}
\max\{q_\gamma',\frac{\gamma'+\frac{n}{r}-1}{\nu},\frac{\gamma'-\frac{n}{r}+1}{\nu-\frac{n}{r}+1}\}<s<q_\gamma\,\,\,\textrm{and}\,\,\,\max\{q_\beta',\frac{\beta+\frac{n}{r}-1}{\nu},\frac{\beta-\frac{n}{r}+1}{\nu-\frac{n}{r}+1}\}< q<q_\beta;
\end{equation}
or
\begin{equation}\label{weakfactocond2}
1<s<2\,\,\,\textrm{and}\,\,\,1<q\leq 2.
\end{equation}
Then the following assertions are equivalent.
\begin{itemize}
\item[(i)] $A_\gamma^s(\mathcal{D})=A_\alpha^p(\mathcal{D})\bigotimes A_\beta^q(\mathcal{D})$.
\item[(ii)] For any analytic function $b$ on $\mathcal{D}$, $h_b^{(\nu)}$ extends as a bounded operator from $A_\alpha^p(\mathcal{D})$ to $\mathbb{B}_{\beta'}^{q'}(\mathcal{D})$
if and only if $b$ is a representative of class in $\mathbb{B}_{\gamma'}^{s'}(\mathcal{D}).$
\end{itemize}
\etheo
\begin{proof}
Recall that $\gamma'=\gamma+(\nu-\gamma)s'$, $\beta'=\beta+(\nu-\beta)q'$. When condition (\ref{weakfactocond1}) holds, we have $\gamma',\beta'>\frac{n}{r}-1$ and $\mathbb{B}_{\beta'}^{q'}(\mathcal{D})=A_{\beta'}^{q'}(\mathcal{D})$ and $\mathbb{B}_{\gamma'}^{s'}(\mathcal{D})=A_{\gamma'}^{s'}(\mathcal{D})$.
\vskip .2cm
That $(\textrm{i})\Rightarrow (\textrm{ii})$ is harmless. To prove
$(\textrm{ii})\Rightarrow (\textrm{i}),$ we start by establishing the following:

\blem\label{facto} Let $\gamma,\nu>\frac nr-1$ and $1<s<q_\gamma.$ Let $1<p, q<\infty$ and $\alpha,\beta>\frac nr-1$
so that $\frac{1}{s}=\frac{1}{p}+\frac{1}{q}<1$, $\frac{\gamma}{s}=\frac{\alpha}{p}+\frac{\beta}{q}$ and the hypotheses in Theorem \ref{prop:equivhankweakfact} are satisfied.

Assume that for any analytic function $b$ on $\mathcal{D}$, $h_b^{(\nu)}$ extends as a bounded operator from $A_\alpha^p(\mathcal{D})$ to $\mathbb{B}_{\beta'}^{q'}(\mathcal{D})$
if and only if $b\in \mathbb{B}_{\gamma'}^{s'}(\mathcal{D}).$ Then $$A_\alpha^p(\mathcal{D})\bigotimes A_\beta^q(\mathcal{D})\subset A_\gamma^s(\mathcal{D}).$$
\elem

\Proofof{Lemma \ref{facto}}

Following the reasoning of \cite{PauZhao2}, let $F\in (A_\gamma^s(\mathcal{D}))^*.$ Thanks to Lemma \ref{lem:ApBp0}, there is $b\in \mathbb{B}^{s'}_{\gamma'}(\mathcal{D})$ and $m$ a large enough integer such that $F(\varphi)=\langle b,\varphi\rangle_{\nu,m}$ for all $\varphi\in A_\gamma^s(\mathcal{D}).$ Let $f\in A_\alpha^p(\mathcal{D})\bigotimes A_\beta^q(\mathcal{D}).$
For every $\varepsilon>0,$ there exist finite sequences $\{g_j\}\in A_\alpha^p(\mathcal{D})$ and $\{l_j\}\in A_\beta^q(\mathcal{D})$ such that $f=\sum\limits_jg_jl_j$ and
$\sum\limits_j\|g_j\|_{p,\alpha}\|l_j\|_{p,\beta}\leq \|f\|+\varepsilon.$ Using the reproduction formula (\ref{repr2}) and the boundedness of $h_b^{(\nu)},$ we write
\Beas
|F(f)|&=&|\langle b, f\rangle_{\nu,m}|=|\langle b, \sum\limits_jg_jl_j\rangle_{\nu,m}|=|\sum\limits_j\langle b, g_jl_j\rangle_{\nu,m}|
=|\sum\limits_j\langle \bar{g}_j\Box^mb,l_j\rangle_{\nu+m}|\\ &=& |\sum\limits_j\langle \Box^mh_b^{(\nu)}({g}_j),l_j\rangle_{\nu+m}|\\ &=& |\sum\limits_j\langle h_b^{(\nu)}({g}_j),l_j\rangle_{\nu,m}|\\
&\leq&\sum\limits_j|\langle h_b^{(\nu)}({g}_j),l_j\rangle_{\nu,m}|\leq
\sum\limits_j\|h_b^{(\nu)}({g}_j)\|_{\mathbb{B}_{\beta'}^{q'}}\|l_j\|_{q,\beta}\\&\leq&\|h_b\|\sum\limits_j\|{g}_j\|_{p,\alpha}\|l_j\|_{q,\beta}<\|h_b\|(\|f\|+\varepsilon).
\Eeas
Letting $\varepsilon$ tend to $0$ yields $|F(f)|\leq\|h_b\|\|f\|,$ that is $F\in (A_\alpha^p(\mathcal{D})\bigotimes A_\beta^q(\mathcal{D}))^*.$
\ProofEnd
To prove the reverse inclusion we need the following:
\blem\emph{(\cite[Proposition 5.1]{BBPR})}\label{facto 1}
Assume that $1\leq s<\infty$ and $\gamma,\sigma>\frac nr-1$ and $P_\sigma$ is bounded on $L_\gamma^s(\mathcal{D}).$ Then $A^s_\gamma(\mathcal{D})$ is the closed linear span
of the set $\{B_\sigma (\cdot,z):z\in \mathcal{D}\}.$ In particular, $B_\sigma (\cdot,z)\in A_\gamma^s(\mathcal{D})$ for any $z\in \mathcal{D}$.
\elem
We then deduce the reverse inclusion based on the Hahn-Banach Theorem.
\blem \label{facto 2} Let $\gamma>\frac nr-1$ and $1<s<q_\gamma.$ Let $1<p, q<\infty$ and $\alpha,\beta>\frac nr-1$
so that $\frac{1}{s}=\frac{1}{p}+\frac{1}{q}<1$, $\frac{\gamma}{s}=\frac{\alpha}{p}+\frac{\beta}{q}$.
The space $A_\alpha^p(\mathcal{D})\bigotimes A_\beta^q(\mathcal{D})$ is dense in $A^s_\gamma(\mathcal{D})$ and consequently
$$(A_\alpha^p(\mathcal{D})\bigotimes A_\beta^q(\mathcal{D}))^*\subset (A^s_\gamma(\mathcal{D}))^*.$$
\elem
\Proofof{Lemma \ref{facto 2}}
As $1<s<q_\gamma$, taking $\sigma=m_1+m_2$ where $m_1$ and $m_2$ are two large positive numbers, we have that by Lemma \ref{lem:bergtype}
that $P_\sigma$ is bounded on $A_\gamma^p(\mathcal{D})$. It follows from Lemma \ref{facto 1} that the space $A_\gamma^s(\mathcal{D})$ is 
the closed linear span of the set $\{B_{m_1+m_2} (\cdot,z):z\in \mathcal{D}\}.$ Thus, for every $f\in A_\gamma^s(\mathcal{D}),$ there
is a sequence
$\{f_k\}$ defined by $f_k=\sum\limits_{j=1}^ka_jB_{m_1+m_2}(\cdot, z_j),$ where $a_j$'s are complex constants,
such that $\lim\limits_{k\to\infty}\|f-f_k\|_{s,\gamma}=0.$ Observe
now that $B_{m_1+m_2}(\cdot, z_j)=C_\gamma\Da^{-m_1-m_2-\frac nr}\left(\frac{\cdot-\bar{z}_j}{i}\right)=g_jl_j$
where $g_j=C_\gamma\Da^{-m_1}\left(\frac{\cdot-\bar{z}_j}{i}\right)$ and
$l_j=\Da^{-m_2-\frac nr}\left(\frac{\cdot-\bar{z}_j}{i}\right)$. For $m_1$ and $m_2$ sufficiently large, we also have that
$g_j\in A_\alpha^p(\mathcal{D})$ and $l_j\in A_\beta^q(\mathcal{D})$ for any $j=1,2,\cdots,k$, and 
$\|B_\gamma(\cdot, z_j)\|_{s,\gamma}=\|g_j\|_{p,\alpha}\|l_j\|_{q,\beta}.$
Hence for any $k\in\N,$
$f_k=\sum\limits_{j=1}^ka_jg_jl_j$ with $g_j\in A^p_\alpha(\mathcal{D})$ and $l_j\in A^q_\beta(\mathcal{D}).$ This shows that the sequence
$\{f_k\}$ lies in $A_\alpha^p(\mathcal{D})\bigotimes A_\beta^q(\mathcal{D})$ and therefore we are done with the density.

Furthermore, if $F\in (A_\alpha^p(\mathcal{D})\bigotimes A_\beta^q(\mathcal{D}))^*,$ by the Hanh-Banach Theorem, there $G\in (A^s_\gamma(\mathcal{D}))^*$ that extends $F$ with
$\|G\|=\|F\|.$ By the density of $A_\alpha^p(\mathcal{D})\bigotimes A_\beta^q(\mathcal{D})$ in $A^s_\gamma(\mathcal{D})$ and the boudedness of both $F$ and $G,$ we conclude
that $F=G\in (A^s_\gamma(\mathcal{D}))^*.$

\ProofEnd
The proof is complete.
\end{proof}
For $\nu>\frac{n}{r}-1$, let us denote by $\mathcal{R}_\nu$ the set of exponent $p\in [1,\infty)$ such that $P_\nu$ is bounded on $L_\nu^p(\mathcal{D})$. We have the following result.
\btheo
Let $\nu>\frac{n}{r}-1$. Assume that $p,q,s\in \mathcal{R}_\nu$ and $\frac{1}{s}=\frac{1}{p}+\frac{1}{q}<1$. Then the following assertions are equivalent.
\begin{itemize}
\item[(i)] $A_\nu^s(\mathcal{D})=A_\nu^p(\mathcal{D})\bigotimes A_\nu^q(\mathcal{D})$.
\item[(ii)] For any analytic function $b$ on $\mathcal{D}$, $h_b^{(\nu)}$ extends as a bounded operator from $A_\nu^p(\mathcal{D})$ to $A_{\nu}^{q'}(\mathcal{D})$
if and only if $b$ is a representative of class in $\mathbb{B}_{\nu}^{s'}(\mathcal{D}).$
\end{itemize}
\etheo
\begin{proof}
We only have to take care of the implication $(ii)\Rightarrow (i)$. This follows as above with the following lemma in place of Lemma \ref{facto 2}.
\blem\label{lem:facto3}
Let $\nu>\frac{n}{r}-1$. Assume that $p,q,s\in \mathcal{R}_\nu$ and $\frac{1}{s}=\frac{1}{p}+\frac{1}{q}<1$. Then the space $A_\nu^p(\mathcal{D})\bigotimes A_\nu^q(\mathcal{D})$ is dense in $A^s_\nu(\mathcal{D})$ and consequently,
$$(A_\nu^p(\mathcal{D})\bigotimes A_\nu^q(\mathcal{D}))^*\subset (A^s_\nu(\mathcal{D}))^*.$$
\elem
\begin{proof}[Proof of Lemma \ref{lem:facto3}]
Let $F\in A^s_\nu(\mathcal{D})$, $\alpha_1,\alpha_2>0$, and $\epsilon>0$. Put $\alpha=\alpha_1+\alpha_2$, and define $$F_{\epsilon,\alpha}(z)=F(z+i\epsilon e)\Delta^{-\alpha}(\frac{\epsilon z+ie}{i}).$$
We have the following facts (see the proof of \cite[Theorem 3.23]{BBGNPR}):
\begin{itemize}
\item[(a)] $F_{\epsilon,\alpha}\in A^s_\nu(\mathcal{D})$ and $\|F_{\epsilon,\alpha}\|_{s,\nu}\leq \|F\|_{s,\nu}$.
\item[(b)] $\lim_{\epsilon\rightarrow 0}\|F-F_{\epsilon,\alpha}\|_{s,\nu}=0$.
\item[(c)] For $\alpha_1$ large enough, $F_{\epsilon,\alpha_1}\in A^p_\nu(\mathcal{D})$.
\end{itemize}
Observe also with Lemma \ref{int-compl} that for $\alpha_2$ large enough, $\Delta^{-\alpha_2}(\frac{\epsilon z+ie}{i})$ is  in $A^q_\nu(\mathcal{D})$. Thus $F_{\epsilon,\alpha}=F_{\epsilon,\alpha_1}G_{\epsilon,\alpha_2}$, where $G_{\epsilon,\alpha_2}(z):=\Delta^{-\alpha_2}(\frac{\epsilon z+ie}{i})$, with $\|F_{\epsilon,\alpha}\|_{s,\nu}\leq \|F_{\epsilon,\alpha_1}\|_{p,\nu}\|G_{\epsilon,\alpha_2}\|_{q,\nu}$. Hence $A_\nu^p(\mathcal{D})\bigotimes A_\nu^q(\mathcal{D})$ is dense in $A^s_\nu(\mathcal{D})$.
\vskip .2cm
That $(A_\nu^p(\mathcal{D})\bigotimes A_\nu^q(\mathcal{D}))^*\subset (A^s_\nu(\mathcal{D}))^*$ follows as in the proof of Lemma \ref{facto 2}.
\end{proof}
The proof is complete.
\end{proof}
\bibliographystyle{plain}

\begin{thebibliography}{1}






\bibitem{Be}
\textsc{B\'ekoll\'e, D.,}
\newblock{``The dual of the Bergman space $A^1$ in symmetric Siegel domains of type II''}.
\newblock{Trans. Amer. Soc. 296(2); 607-619 Aug. 1986.}

\bibitem{BB1}
\textsc{B\'ekoll\'e, D., A. Bonami.}
\newblock {``Estimates for the
Bergman and Szeg\H o projections in two symmetric domains
of $\mathbb C^n$''.}
\newblock {Colloq. Math. {\bf 68} (1995), 81-100.}

\bibitem{BB2}
\textsc{B\'ekoll\'e, D., A. Bonami.}
\newblock {``Analysis on tube domains over light cones: some extensions of recent results''}.
\newblock{{\it Actes des Rencontres d'Analyse Complexe, Poitiers 1999.} \'Ed.. Atlantique et ESA CNRS 6086
(2000), pp. 17-37.}




\bibitem{BBGNPR}
\textsc{D. B\'ekoll\'e,  A. Bonami, G. Garrig\'os, C. Nana,
M. Peloso and F. Ricci.}
\newblock {\emph{Lecture notes on Bergman projectors in tube domains over cones: an analytic and geometric viewpoint}.}
\newblock {IMHOTEP 5 (2004), Expos\'e I, Proceedings of the International Workshop in Classical Analysis, Yaound\'e
2001.}



\bibitem{BBGR}
\textsc{D. B\'ekoll\'e, A. Bonami, G. Garrig\'os and
F. Ricci.}
\newblock {``Littlewood-Paley decompositions related to symmetric cones and Bergman projections in tube domains''}.
\newblock {Proc. London Math. Soc. {\bf 89} (2004), 317-360.}

\bibitem{BBGRS}
\textsc{D. B\'ekoll\'e,  A. Bonami, G. Garrig\'os, F.Ricci
and B. Sehba .}
\newblock {``Hardy-type inequalities and analytic Besov spaces in tube domains over symmetric cones''}.
\newblock {J. Reine Angew. Math. {\bf 647} (2010), 25-56.}



\bibitem{BBPR}
\textsc{D. B\'ekoll\'e, A. Bonami, M. Peloso and F.
Ricci.}
\newblock {``Boundedness of weighted Bergman projections on tube domains over light cones''}.
\newblock {Math. Z. {\bf 237} (2001), 31-59.}

\bibitem{BIN}
\textsc{D. B\'ekoll\'e., H. Ishi and C. Nana},
\textit{Kor\'anyi's Lemma for homogeneous Siegel domains of type II. Applications and extended results},
Bull. Aust. Math. Soc. \textbf{90} (2014), 77-89.

\bibitem{BT98}
\textsc{B\'ekoll\'e, D. and A. Temgoua Kagou.}
\newblock {``Molecular decompositions and interpolation''}.
\newblock {integr. equ. oper. theory}
\newblock {(31)}
\newblock {(1998), 150-177}.


\bibitem{B}
\textsc{ A. Bonami}
\newblock {``Three related problems on Bergman spaces over symmetric cones''}.
\newblock {Rend. Mat. Acc. Lincei s. 9, v. 13 (2002), 183-197.}


\bibitem{BL}
\textsc{A. Bonami and L. Luo}
\newblock{On Hankel operators between Bergman spaces on the
unit ball}.
\newblock{Houston J. Math. Vol. 31, no. 3
(2005), 815-828.}



\bibitem{Constantin}
\textsc{O. Constantin},``Carleson embeddings and some classes of operators on weighted
Bergman spaces''.J. Math. Anal. Appl. {\bf 365} (2010) 668-682.

\bibitem{DD}
\textsc{D. Debertol,}
\newblock{\em Besov spaces and boundedness of weighted Bergman projections over symmetric tube domains,}
\newblock{Dottorato di Ricerca in Matematica,{\it Universit\`a di Genova, Politecnico di Torino}, (April 2003).}


\bibitem{FK}
\textsc{J. Faraut,  A. Kor\'anyi.}
\newblock {{\it Analysis on symmetric cones}}.
\newblock {Clarendon Press, Oxford, (1994).}




\bibitem{GS} \textsc{G. Garrig\'os, A. Seeger}, ``Plate
decompositions for cone multipliers''. Proc. ``Harmonic
Analysis and its Applications at Sapporo 2005'', Miyachi \&
Tachizawa Ed. Hokkaido University Report Series {\bf 103},
pp. 13-28.

\bibitem{G}
\textsc{J. Gonessa}, Espaces de type Bergman dans les domaines
homog\`enes de Siegel de type II: D\'ecomposition atomique des
espaces de Bergman et Interpolation. Th\`ese de Doctorat/Ph.D,
Universit\'e de Yaound\'e I (2006).

\bibitem{Gowda}
\textsc{M. Gowda}, Nonfactorization theorems in weighted Bergman and Hardy spaces on the unit ball of Cn, Trans.
Amer. Math. Soc. {\bf 277} (1983), 203-212.

\bibitem{Horowitz}
\textsc{C. Horowitz}, Factorization theorems for functions in the Bergman spaces, Duke Math. J., {\bf 44} (1977), 201-213.

\bibitem{Luecking1}
\textsc{D. Luecking}, ``Embedding theorems for spaces of analytic functions via Khinchine抯 inequality'', Michigan Math. J. {\bf 40} (2) (1993) 333-358.

\bibitem{Luecking2}
\textsc{D. Luecking}, ``Trace ideal criteria for Toeplitz operators'', J. Funct. Anal. {\bf 73} (2) (1987) 345-368.



\bibitem{NaSeh}
\textsc{C. Nana, B. F. Sehba}, Carleson embeddings and two operators on Bergman spaces of tube domains over symmetric cones. Available as http://arxiv.org/abs/.

\bibitem{PauZhao1}
\textsc{J. Pau, R. Zhao}, Carleson measures and Toeplitz operators for weighted Bergman spaces of the unit ball. Available as http://arxiv.org/abs/1401.2563v1.

\bibitem{PauZhao2}
\textsc{J. Pau, R. Zhao}, Weak factorization and Hankel form for weighted Bergman spaces on the unit ball. Available as http://arxiv.org/abs/1407.4632v1.


\bibitem{RS}
\textsc{R. Rochberg, S. Semmes}, Nearly weakly orthonormal sequences, singular values estimates, and Calderon-
Zygmund operators, J. funct. Analysis {\bf 86}, 237-306 (1989).

\bibitem{sehba}
\textsc{B. F. Sehba}
\newblock {``Bergman type operators in tubular domains over
symmetric cones"}, Proc. Edin. Math. Soc. {\bf 52} (2) (2009), 529--544.

\bibitem{S}
\textsc{B. F. Sehba}
\newblock {``Hankel operators on Bergman spaces of tube domains over
symmetric cones"}.  Integr. equ. oper. theory {\bf 62}
(2008), 233--245.













\bibitem{Zhu1}
\textsc{K. Zhu}
\newblock{\it Operator theory in function spaces.}
\newblock  {Marcel Dekker, New York 1990.}





\end{thebibliography}

\end{document}